\documentclass[11pt, english]{article} 

\pdfoutput=1
\usepackage{color}
\usepackage[T1]{fontenc}
\usepackage[margin=1in]{geometry} 
\usepackage{graphicx} 
\usepackage{amsmath} 
\usepackage{amsthm} 
\usepackage{mathrsfs}
\usepackage{mathtools}
\usepackage{graphicx}
\usepackage{enumitem}
\usepackage{amsmath}
\usepackage{wrapfig}
\usepackage{amsfonts}
\usepackage{babel}
\usepackage{pst-node}
\usepackage{tikz-cd} 
\usepackage{titlesec}
\usepackage[colorlinks]{hyperref}
\titleformat*{\section}{\normalsize\bfseries}
\titleformat*{\subsection}{\small\bfseries}
\titleformat*{\subsubsection}{\normalize\bfseries}
\titleformat*{\paragraph}{\normalize\bfseries}
\titleformat*{\subparagraph}{\normalize\bfseries}

\usepackage{color}  
\usepackage{xcolor}

\usepackage{hyperref}
\newtheorem{theorem}{Theorem}[section]
\newtheorem{lemma}[theorem]{Lemma}

\theoremstyle{definition}
\newtheorem{definition}[theorem]{Definition}
\theoremstyle{remark}
  
\newtheorem{remark}[theorem]{Remark}

\newcommand{\Z}{\mathbb{Z}}

\DeclareMathAlphabet{\mathpzc}{OT1}{pzc}{m}{it}

\definecolor{byzantine}{rgb}{0.74, 0.2, 0.64}

\title{Floer Homology:  From Generalized Morse-Smale Dynamical Systems to Forman's Combinatorial Vector Fields }
\author{Marzieh Eidi $^{1, *}$ \and J\"urgen Jost $^{1, 2}$}
\date{%
    $^1$ Max-Planck Institut for Mathematics in the Sciences, Leipzig, Germany\\%
    $^2$ Santa Fe Institute, Santa Fe, New Mexico, USA\\[2ex]
        $^*$ meidi@mis.mpg.de}

\begin{document}
\maketitle
\begin{center}
    \today
\end{center}
\begin{abstract}
We construct a  Floer type boundary operator for generalised Morse-Smale dynamical systems on compact smooth manifolds by counting the number of suitable flow lines between closed (both homoclinic and periodic) orbits and isolated critical points.  The same principle works for the discrete situation of general combinatorial vector fields, defined by Forman, on CW complexes.  We can thus recover the $\mathbb{Z}_2$ homology of both smooth and discrete structures directly from the flow lines (V-paths) of our vector field. 

\end{abstract}

\section{Introduction}
One of the key ideas of modern geometry is to extract topological information
about an object from a dynamical process operating on that object. For that
purpose, one needs to identify the invariant sets  and the dynamical relations
between them. The invariant sets generate groups, and the dynamics defines
boundary operators, and when one has shown that these operators square to zero,
one can then define homology groups. The first such ideas may be seen in the
works of Riemann, Cayley and Maxwell in the 19th century. In 1925, Morse
\cite{Morse1925} developed his famous theory where he recovered the homology
of a compact Riemannian manifold $M$ from the critical points of a smooth
function $f$, assuming that these critical points are all non-degenerate. The
dynamics in question is that of the gradient flow of $f$. The basic invariant sets
then are precisely the critical points of $f$. The theory was analyzed and
extended by Milnor, Thom, Smale, Bott and others. In particular, Bott
\cite{Bott} extended
the theory to the case where the gradient of $f$ is allowed to vanish on a
collection of smooth submanifolds of $M$.

Based on ideas from supersymmetry,
Witten constructed an interpolation between de Rham and Morse homology. Floer \cite{Floer89}
then developed the very beautiful idea that the boundary operator in Morse
theory can be simply obtained from counting gradient flow lines (with
appropriate orientations) between critical points of index difference
one. The first systematic exposition of Floer's ideas was given in
\cite{Schwarz93} (see also \cite{Jost17a}). The main thrust of Floer's work was devoted to infinite dimensional
problems around the Arnold conjecture, see
\cite{Floer88a,Floer88b,Floer88c,Floer88d}, because for his theory, he only
needed relative indices, and not absolute ones, so that the theory could be
applied to indefinite action functionals. But also in the original finite
dimensional case, Floer's theory advanced our insight considerably and
motivated much subsequent work.

In fact, Floer \cite{Floer89} had been  motivated by another beautiful theory
relating dynamics and topology, that of Conley \cite{Conley78} (for more
details, see \cite{Conley/Zehnder84} and for
instance  the presentations in \cite{Smoller94,Jost2005}). Conley's
theory applies to arbitrary dynamical systems, not just gradient flows.
Actually, Smale \cite{Smale} had already extended Morse theory to an important
and general class of dynamical systems on compact Riemannian manifolds, those
that besides isolated critical points are also allowed to have non-degenerate
closed orbits. Similar to Morse functions, the class of Morse-Smale dynamical
systems is structurally stable, that is, preserves its qualitative properties
under small perturbations. It turns out, however, that these systems can also be subsumed
under Morse-Bott theory. In fact, in  \cite{Smale2} Smale proved that for
every gradient-like system there exists an energy function that is decreasing
along the trajectories of the flow, and  Meyer \cite{Meyer} generalized this
result to the Morse-Smale dynamical systems and defined a Morse-Bott type
energy function based on the flows. Such an energy function would then be
constant on the periodic orbits, and they can then be treated as critical
submanifolds via Morse-Bott theory.  Banyaga and Hurtubise \cite{Banyaga04,Banyaga,hurtubise}
then constructed a general boundary operator for Morse-Bott systems that put
much of the preceding into perspective (see also the detailed literature
review in \cite{hurtubise}). 

There is still another important extension of Morse theory, the combinatorial
Morse
theory of  Forman \cite{Forman} on simplicial and cell
complexes. Here, a function assigns a value to every simplex or cell, and
certain inequalities between the values on a simplex and on its facets are
required that can be seen as analogues of the non-degeneracy conditions of
Morse theory in the smooth setting. As shown in \cite{Benedetti}, this theory
recovers  classical Morse theory by considering PL triangulations of
manifolds that admit Morse functions. This theory has found various practical
applications in diverse fields, such as computer graphics, networks and sensor
networks analysis, homology computation, astrophysics, neuroscience,
denoising, mesh compression, and topological data analysis. (For more details
on smooth and discrete Morse theory  and their applications see
\cite{Knudson,Scoville}). In \cite{Forman2}, Forman also extended his
theory to combinatorial vector fields. 

In the first chapter of this part, we extend Floer's theory into the direction of Conley's
theory. More precisely, we shall show that one can define a boundary operator
by counting suitable flow lines not only for Morse functions, but also for
Morse-Smale dynamical systems, and in fact, we more generally also allow for
certain types of homoclinic orbits in the dynamical system. Perhaps
apart from this latter small extension, our results in the smooth setting
readily follow from the existing literature. One may invoke \cite{Meyer} to treat it as a Morse-Bott
function with the methods of \cite{hurtubise}. Alternatively,  one may  locally perturb
the periodic orbits into heteroclinic ones between two fixed points by a
result of Franks \cite{Franks} and then use \cite{Smale2} to treat it like a
Morse function. In some sense, we are also using such a perturbation. Our
observation then is that the collection of gradient flow lines between the
resulting critical points has a special structure which in the end will allow
us to directly read off the boundary operator from the closed (or homoclinic)
orbits and the critical points. It remains to develop Conley theory in more
generality from this perspective.  Moreover,   our approach also readily
extends to the combinatorial situation of \cite{Forman,Forman2}. Again, it is
known how to construct a Floer type boundary operator for a combinatorial
Morse function, and an analogue of
Witten's approach had already been developed  in
\cite{Forman98c}. Our construction here, however, is different from that of
that paper. In fact, \cite{Forman98c} requires stronger assumptions on the
function than the Morse condition, whereas our construction needs no
further assumptions. What we want to advocate foremost, however, is that the
beauty of Floer's idea of counting flow lines to define a boundary operator
extends also to dynamical systems with periodic orbits, in both the smooth and
the combinatorial setting, and that a unifying perspective can be developed.\\
We should point out that in this part and for both of the next chapters, we only treat homology with $\Z_2$
coefficients. Thus, we avoid having to treat the issue of orientations of
flow lines. This is, however, a well established part of the theory, see \cite{Floer93}
 or also the presentations in \cite{Schwarz93,Jost17a}.
 The structure of this paper is as follows: \\
 In Section 2 after reviewing basic notions, we introduce the generalized
 Morse-Smale systems and we define a boundary operator based on these systems
 by which we can compute the homology groups. In section 3 we turn to the
 discrete settings and present such a  boundary operator for general
 combinatorial vector fields by which we can compute the  homology of finite
 simplicial complexes. Both  sections finish with some
 concrete examples of computing homology groups based on our Floer type
 boundary operators.

\section{Generalized smooth Morse-Smale vector fields}
 \subsection{Preliminaries}
We  consider a smooth $m$ dimensional manifold $M$  that is closed, oriented
and equipped with a Riemannian metric whose distance function we denote by $d$. Let $X$ be a smooth vector field
on $M$ and $\phi_t : M\longrightarrow M $ be the flow associated to $X$. We
first recall some basic terminology.  For
$p\in M$, $\gamma(p) = \cup_{t} \phi_t(p)$ will denote the trajectory of $X$
through $p$. Then for each $p \in M$ we define the limit sets of  $\gamma(p)$
as
{ \begin{eqnarray*}
  \alpha(p) &:=& \cap_ {s \leq 0} \overline{ \cup_ {t \leq s} \phi_t(p)}\\
  \omega(p) &:=&
  \cap_ {s \geq 0}  \overline{ \cup_ {t \geq s} \phi_t(p)} .
  \end{eqnarray*}}
  \begin{definition}
    If $f: M\rightarrow M$ is a diffeomorphism, then $x \in M $ is called
    \emph{chain recurrent} if for any $\varepsilon > 0$  there exist points $x_1 = x,
    x_2,... ,x_n= x $ ($n$ depends on $\varepsilon$) such that $d(f (x_i),
    x_{i+1}) < \varepsilon$ for $ 1 \leq  i \leq n$. For a flow $\phi_t $,  $x \in M $ is 
    \emph{chain recurrent} if for any $\varepsilon > 0$  there exist points $x_1 = x,
    x_2,... ,x_n= x $ and real numbers $t(i)\geq 1$ such that $d( \phi_{t(i)}(x_i),
    x_{i+1}) < \varepsilon$ for $ 1 \leq  i \leq n$.  The set of chain
    recurrent points is called the \emph{chain recurrent set}  and will be denoted by $R(X)$. 
  \end{definition} 
   The \emph{chain recurrent set} $R(X)$ is a closed
   submanifold of $M$ that is invariant under $\phi_t$.  We can think of
   $R(X)$ as the points which come within $ \varepsilon $ of being periodic
   for every $ \epsilon > 0 $. A Morse-Smale dynamical system, as  introduced
   by Smale \cite{Smale},  has the
   fundamental property that it does not have any complicated recurrent
   behaviour and the $\alpha$ and $\omega$ limit sets of every trajectory can
   only be  isolated critical points $p$ or  periodic orbits $O$. Morse-Smale
   dynamical systems are the  simplest structurally stable
   types of dynamics; { that is, if $X$ is Morse-Smale and
     $X'$ is a sufficiently small $C^1$ perturbation of $X$ then there is a
     homeomorphism $h: M\rightarrow M$ carrying orbits of $X$ to orbits of
     $X'$ and preserving their orientation. (Such a homeomorphism is called a topological conjugacy and we say that the two vector fields or their corresponding flows are topologically conjugate)}. Here, we shall consider a somewhat 
   more general case where we allow for a certain type of homoclinic rest
   points and their homoclinic orbits.
 {\begin{definition}
 A periodic orbit of the flow $\phi_t$ on $M$ is hyperbolic if the tangent bundle of $M$ restricted to $O$, $TM_O$, is the sum of three derivative $D\phi_t$ invariant bundles  $E^c\oplus E^u \oplus E^s$ such that: 
   \begin{itemize}
   \item[1.] $E^c$ is spanned by the vector field $X$, tangent to the flow.
 \item[2.] There are constants $C, \lambda >0$, such that $\| D\phi_t(v) \|$  $\geq$ $Ce^{\lambda t} \| v \|$ for $v\in E^u$, $t \geq  0$ and  $\| D\phi_t(v) \|$ $ \leq$  $C^{-1}e^{-\lambda t} \| v \|$ for $v \in E^s, t \geq 0$ where $\|. \|$ is some Riemannian metric.   
   \end{itemize}
   A rest (also called critical)  point $p$ for a flow  $\phi_t$ is called hyperbolic provided that $T_pM= E^u \oplus E^s $ and the above conditions are valid for $v \in E^u$ or $E^s$. 
   \end{definition} }
  {  The stable and unstable manifolds of a hyperbolic periodic orbit $O$, are defined by: \\ 
 $W^{s}(O)= \lbrace x\in M \mid  d(\phi_tx, \phi_ty) \rightarrow 0 \text{ as }
   t\rightarrow\infty$ for some $y\in O \rbrace$ and $W^{u}(O)= \lbrace x\in M \mid  d(\phi_tx, \phi_ty)
   \rightarrow 0 \text{ as } t\rightarrow -\infty$ for some $y\in O \rbrace$. And for a rest point and a homoclinic orbit, we define the stable and unstable manifolds analogously. Also the index of a rest point or a closed orbit is defined to be the dimension of $E^u$.        
   We denote an arbitrary point in a
   homoclinic orbit $H$ by  $H_k^0$  where $k$ is the index of the homoclinic
   orbit $H$ and the homoclinic orbit itself is denoted by $H_k^1$ as it is
   homeomorphic to a circle and therefore is one-dimensional. Similarly by
   $O_k^0$  we mean an arbitrary point in a periodic orbit $O$ of index $k$
   and by $O_k^1$ we mean the orbit itself as a one dimensional structure,
   homeomorphic to a circle. 
 In the following definition we extend the definition of Morse-Smale vector fields.  }  
  { \begin{definition}
We call a  smooth flow $\phi_t$ on $M$ \emph{generalised Morse-Smale} if :
 \begin{enumerate}
 \item The chain recurrent set of the flow consists of a  finite number of hyperbolic rest
   points  $\beta_1(p)$,... $\beta_k(p)$ and/or hyperbolic periodic orbits $\beta_{k+1}(O)$,... $\beta_n(O)$.
 \item  $R(X)$ furthermore may have a finite number of homoclinic orbits  $\beta_{n+1}
   (H)$,... $\beta_l (H)$ 
   that can be obtained via local bifurcation from hyperbolic periodic orbits  $\beta_{n+1}(O),...., \beta_l (O) $.
     \item For each $\beta_i (p), 1 \leq i \leq k $ and each $\beta_i (O), k+1
       \leq i \leq l $ the stable and
       unstable manifolds $W^{s} ({\beta_i })$ and  $W^{u} ({\beta_i})$
       associated with $\beta_i$  intersect transversally.  \\ (Here, two  such
       submanifolds  intersect transversally if  for every $x \in  W^{u}({\beta_i}) \cap W^{s} ({\beta_j})$  we have:   $T_x (M) = T_x W^{u} ({\beta_i}) \bigoplus T_ x W^{s} ({\beta_j})$.) 
 \end{enumerate}
  \end{definition}}
  
  Note that the only difference between generalised Morse-Smale flows
  as defined here  and standard Morse-Smale flows is the possible
  existence of homoclinic points and orbits. In the standard case, one simply
  excludes the second condition. \\
 Therefore  a generalised Morse-Smale flow can be perturbed to a corresponding Morse-Smale flow where all of the homoclinic orbits  $\beta_i (H), n+1 \leq i \leq l $, are substituted by  periodic orbits $\beta_i (O)$.
 For any two distinct $\beta_i$ and $\beta_j$ in the above definition we
 consider $W(\beta_i, \beta_ j )= W^{u} (\beta_ i)\cap W^{s} (\beta_ j)$. Then
 based on the transversality condition, this intersection is  either empty (if
 there is no flow line from $\beta_i$ to $\beta_j$ or a submanifod of
 dimension $\lambda_{\beta_i} - \lambda_{\beta_j} + \dim \beta_i$  where the
 index of $\beta_k$ is denoted by $\lambda_{\beta_k}$ \cite{Banyaga}. In the second case the flow $\phi_t$ induces an $\mathbb{R}$-action on $W(\beta_i, \beta_j)= W^{u} (\beta_i)\cap W^{s} (\beta_j)$. Let
     \begin{center}
    $ M(\beta_i, \beta_j)= (W(\beta_i, \beta_j))/ \mathbb{R})$ 
     \end{center}
  be the quotient space by this action of the  flow lines from $\beta_i$ to $\beta_j$. \\
 	  {\begin{remark}   
 	  In the definition of standard Morse-Smale flow the condition (1) above could be replaced by $(1')$:\\
  All periodic orbits and rest points of the flow are hyperbolic and there
  exists a Morse-Bott type energy function (as defined by Meyer \cite{Meyer}).
  	  \end{remark}}

\subsection{The  chain complex for generalized Morse-Smale vector fields}
Suppose $X$ is a generalized Morse-Smale vector field over $M$. To motivate
our construction, we first observe that by a slight extension of a result of
Franks \cite{Franks}, we can replace every periodic or homoclinic orbit by two
non-degenerate critical points, without changing the flow outside some small
neighborhood of that orbit. \\
 \begin{wrapfigure}{r}{0.45\textwidth}
 \includegraphics[width=8 cm]{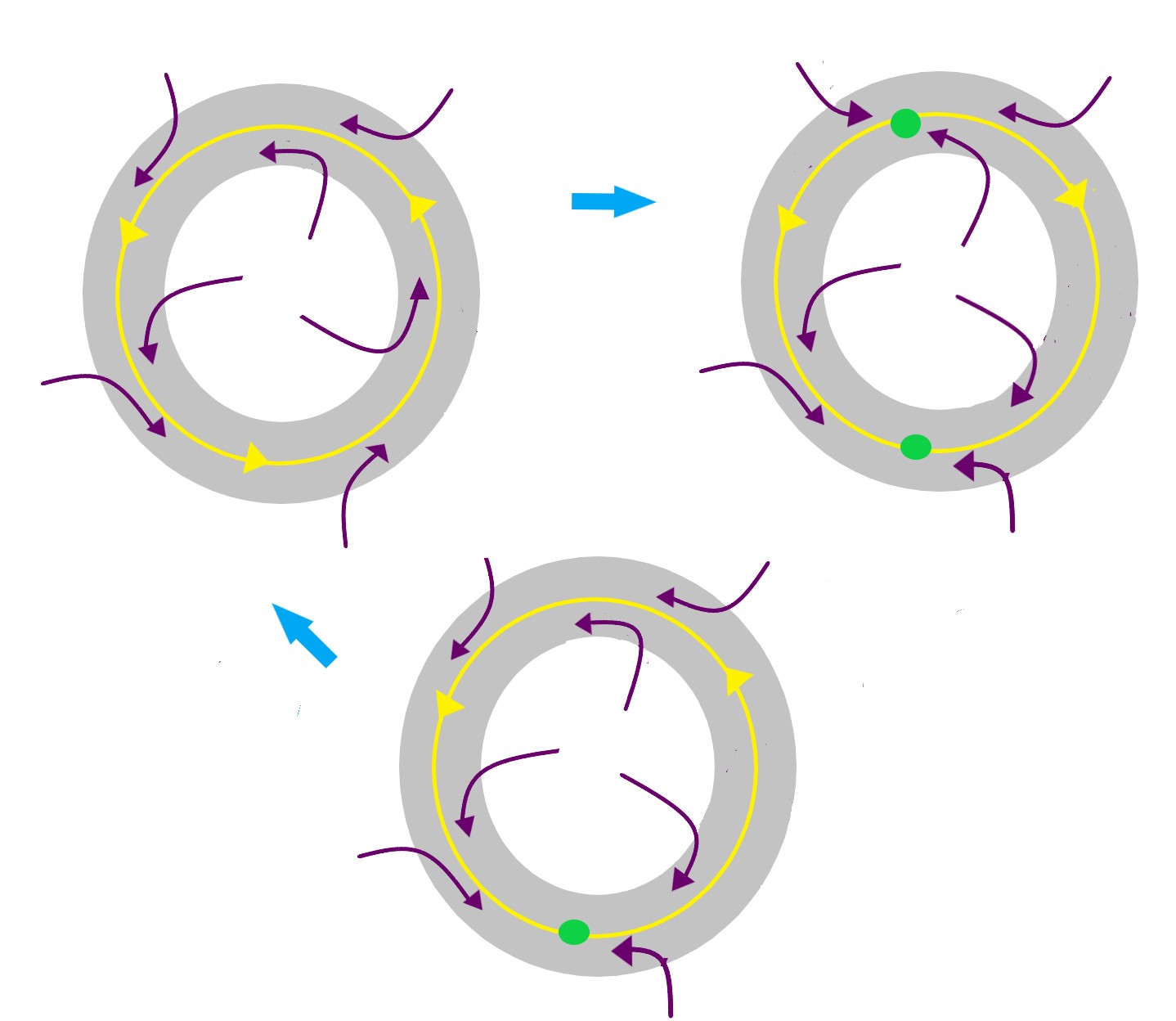} 
\end{wrapfigure}

\begin{lemma}\label{franks}
   Suppose $\phi_t$ is a  generalized Morse-Smale flow on an orientable
   manifold with a periodic or homoclinic orbit of index $k$. Then for any
   neighborhood $U$ of that orbit there exists a new generalized Morse-Smale flow $\phi'_t$ whose vector field agrees with that of $\phi_t$ outside $U$ and which has rest points $t$ and $t'$ of index $k+1$ and $k$ in $U$ but no other chain recurrent points in $U$.  
 \end{lemma}

\begin{proof}
In \cite{Franks} Franks proved that for a Morse-Smale flow $\phi_t$  on an
orientable manifold with a closed periodic orbit $O$ of index $k$ and a given
neighbourhood $U$ of $O$, there exists a new Morse-Smale flow $\phi'_t$ whose
vector field agrees with that of $\phi_t$ outside $U$ and which has rest
points $q_1$ and $q_2$ of index $k$ and $k+1$ in $U$ but no other chain
recurrent points in $U$.

 For the generalized Morse-Smale flow we note that
each homoclinic orbit is by definition obtained in a continuous local
bifurcation of a periodic orbit. 
Therefore if we use this bifurcation in the
reverse direction and substitute again any such homoclinic orbit with its
corresponding periodic orbit we can use Franks' argument for replacing all the
periodic and homoclinic orbits with two rest points and two heteroclinic orbits between them.   
\end{proof}

 	  {\begin{remark}   
In the above figure, the qualitative features of the three cases
outside the gray annulus are 
the same. In particular, we can bifurcate two heteroclinic orbits between two
critical points (in the right) to get a homoclinic orbit and a homoclinic critical point (in the middle) by bringing the two critical points closer and closer and then bifurcate the homoclinic orbit to a periodic orbit (in the left). 
\end{remark}} 

With this lemma, we can turn our flow into one that has only
non-degenerate critical points. We could then simply utilize the Floer
boundary operator for that flow. In fact, that motivates our construction, but
we wish to define a Floer type boundary operator directly in terms of the
periodic and homoclinic orbits and the critical points. Our simple observation
is that a Floer boundary operator resulting from the replacement that Franks proposed,
has some additional structure that is derived from the orbits that have been
perturbed away. This allows for an arrangement of the flow lines between the
critical points of the perturbed flow that leads to the definition of the
boundary operator in the presence of those orbits. That is, we can read off
the boundary operator directly from the relations between the orbits and the
critical points without appealing to that perturbation, although the
perturbation helps us to see why this boundary operator squares to 0. \\
Furthermore the above remark on the reverse of Franks' perturbation procedure can be used to turn a pair of critical points whose indices differ by one and which are connected by two flow lines to a periodic orbit. Similarly for homoclinic orbits.\\
On the other hand, by the Morse cancellation theorem, we can cancel a pair of critical points $(p_k, p_{k-1})$ when there is a single flow line between them (if their unstable and stable manifolds, resp.,  intersect transversally). For more details on this theorem, we refer to \cite{Milnor, cancel}.

We now define the
Morse-Floer complex  $(C_*(X), \partial)$ of $X$ as follows. Let $C_k$ denote the finite vector space (with coefficients in $\mathbb{Z}_2$) generated by the following set of rest points/orbits of the vector field: 
	\[ \left(  p_k,  O^0_ {k},  O^1_ {k-1}, H_ {k}^0, H_ {k-1}^1
          \right) . \]

Before defining a boundary operator for this chain complex we need  one more definition: \\
\begin{definition}\label{3}
    We call a  generalized Morse Smale system   \emph{simple} if  with the help of the Morse cancellation theorem  and/or the 
reverse of Franks' perturbation procedure we can get a generalized Morse-Smale system which has no flow lines between its generators of the same $C_k$. 
\end{definition}
We note that in a (generalized) Morse-Smale in principle we might have flow lines between two orbits of the same index or a critical point of index $k$ to an orbit of index $k-1$ which belong to the same $C_k$.  From now on, we assume that our generalized Morse-Smale system is simple. For simple systems, we define the differential $\partial$ as follows:

The differential $\partial_k : C_k(X) \longrightarrow C_{k-1}(X)$  counts the
number of connected components of $M(\beta_i, \beta_ j )$ (mod 2) where
$\beta_i$ and $\beta_j$ are isolated rest points $p_k$ or closed orbits
(either homoclinic orbits $H$  or periodic orbits  $O$). Here, each such
orbit, carrying topology in two adjacent dimensions, corresponds to two
elements in the boundary calculus. More precisely, a periodic orbit $O_{k}$ of index
$k$ generates an element $O_k^1$ of dimension $k+1$ and an element $O_{k}^0$
in dimension $k$, and analogously for homoclinics. Thus, our boundary
operator is:
          \begin{eqnarray*}
\partial p_k &=& \sum \alpha (p_k, p_{k-1})  p_{k-1} + \sum  \alpha (p_k,  O_{k-2})   O^1_{k-2}\\ &+& \sum  \alpha (p_k,  H_ {k-2})   H_ {k-2}^1  \\\\ \partial  O^0_{k}&=& \sum \alpha ( O_{k},  O_{k-1})  O^0_{k-1}+  \sum \alpha ( O_{k},  H_ {k-1})  H_ {k-1}^0 \\ &+&  \sum \alpha ( O_{k}, p_{k-1}) p_{k-1} \\\\  \partial  O^1_{k-1}&=& \sum \alpha( O_{k-1},  O_{k-2})  O^1_{k-2} + \sum \alpha ( O_{k-1},  H_ {k-2})  H_ {k-2}^1 \\\\  \partial  H_ {k}^0&=& \sum \alpha ( H_ {k},  H_ {k-1}) H_ {k-1}^0+  \sum \alpha ( H_ {k},  O_ {k-1})  O_ {k-1}^0 \\ &+&   \sum \alpha ( H_ {k}, p_{k-1}) p_{k-1} \\\\  \partial  H_ {k-1}^1&=& \sum \alpha ( H_ {k-1}, H_ {k-2})  H_ {k-2}^1+ \sum \alpha ( H_ {k-1},  O_{k-2})  O^1_{k-2}.  
\end{eqnarray*}
In this definition, the sums extend over all the elements on the right hand
side; for instance, the first sum in the first line is over all critical
points $p_{k-1}$ of index $k-1$.  $\alpha (p_k, p_{k-1})$, similar to the classical
Morse-Floer theory (where there is no closed orbit and therefore the vector
field is, { up to topological conjugacy}, gradient-like), is the number of  flow lines from $p_k$ to $p_{k-1}$.
 We observe some terms do not appear; for instance, we do not
have terms with coefficients  of the form { $\alpha (p_k,  O_{k-1}) $,
  {nor of the form} $\alpha (p_k,  H_{k-1})$. This
will be important  below in the
proof of Thm. \ref{1}. The reason why such a term does not show up is that if
there were a flow line from some $p_k$ to some $O_{k-1}^0$, then there would
also be a flow to the corresponding $O_{k-1}^1$ which comes from the same
closed orbit. But $O_{k-1}^1$ and $p_k$ are the elements of the same  $C_k$,
and by the simple generalised Morse-Smale system condition, there are no flow lines between critical
elements of the same $C_k$. Analogously for homoclinics.\\
	  \begin{remark}   
  Note that in defining  the chain complex and the corresponding  boundary
  operator $\partial$ for $X$, we could first replace all the homoclinic
  orbits with bifurcated periodic orbits and present our definitions for the
  simpler case where  all the closed orbits are periodic. Then we would have
  just three generators	\[ \left(  p_k,  O^0_ {k},  O^1_ {k-1}  \right) \] for
  $C_k(X)$. However here we choose not to do this to emphasize that we can
  construct the boundary operator also for homoclinic orbits  as long as our operator is defined based on the flow lines outside the tubular neighborhoods of orbits. 
   \end{remark}    
  \begin{theorem} \label{1}
  $\partial^ 2 = 0$.
   \end{theorem} 
 In classical Morse-Floer theory, one assumes that there are only isolated
critical points and no closed or homoclinic orbits, and therefore  all the $\alpha$ coefficients in the
definition of $\partial $ except the first one (in the first row) are zero; there to
prove $\partial^ 2 = 0$ one can then use the classification theorem of one
dimensional compact manifolds where the number of connected components of
their boundary mod two is zero (see \cite{Jost17a}). Here as  $W(\beta_i,
\beta_ {i-1} )$ might have dimension bigger than one, the number of connected
components of the boundary of compact two dimensional manifolds might
vary. For our generalized Morse-Smale flows, however, we use Lemma \ref{franks} 
to replace any orbit of index $k$ (both periodic $O_ {k}$ and
homoclinic $H_ {k}$)  by a rest point of index $k$ and one of $k+1$
which are joined by two heteroclinic orbits. When replacing
$H_ {k}$,  the resulting rest point of higher index can be taken 
to be the point  $h$ itself, which then will be no longer  homoclinic. 
 \begin{proof} By the above replacement, we get a vector field $Y$
   which has no periodic and homoclinic orbits and is therefore
   gradient-like (up to topological conjugacy \cite{Franks1979}). This $Y$  has all
   the isolated rest points of $X$, two isolated rest points
   $q^{up}_k$ and $q'^{down}_{k-1}$ instead of every periodic orbit
   $O_{k-1}$ of index $k-1$ and two isolated rest points $t^{up}_k$ and
   $ t'^{down}_{k-1}$ instead of every homoclinic orbit $H_ {k-1}$ of
   index $k-1$. We note that all the critical points in
  $Y$ are isolated and for each index $k$ they can be  partitioned into five
  different sets $p_k, t^{up}_k, t'^{down}_k , q^{up}_k, q'^{down}_k $. This
  partitioning is possible because orbits and isolated rest points have pairwise empty 
  intersection.
  The proof will now consist of the following main steps:

      \item[1.]   We define $ C_k(Y)$ to be the finite vector space (with coefficients in $\mathbb{Z}_2$) generated by
   	\[ \left(  p_k,  q^{up}_k,  q'^{down}_{k} ,  t^{up}_k,  t'^{down}_{k} \right) \]
  
       \item[2.] We  define a boundary operator $\partial'$ and consequently a chain complex 
       corresponding to $(Y,C_*(Y),  \partial')$. 
        \item[3.]  and then we prove there is an isomorphism (chain map) $\varphi_* : 
        C_*(X) \longrightarrow C_*(Y)$. Since $\varphi$ is an isomorphism we get our desired equality $\partial^
    2 = 0$ as $\partial = \varphi_* ^{-1}  \partial' \varphi_*$ and  $\partial^ 2=\varphi_* ^ {-1} { 
    \partial'}^2\varphi_*$  
 \begin{itemize}    
 
  \item[1.]  We note that in $ C_k(Y)$,  $q^{up}_k$ comes from a periodic orbit of index $k-1$ and  $q'^{down}_{k} $ comes from the replacement of a periodic orbit of index $k$. Similarly  $t^{up}_k$ is obtained from replacing a homoclinic orbit of index $k-1$ and $t'^{down}_{k}$ comes from a homoclinic orbit of index $k$.
    \item[2.] We define  $\partial'_k : C_k(Y) \longrightarrow C_{k-1} (Y)$ as follows:
      \begin{eqnarray*}
	\partial' p_k &=& \sum \alpha (p_k, p_{k-1})  p_{k-1} + \sum  \alpha(p_k, q^{up}_{k-1})  q^{up}_{k-1} \\ &+& \sum  \alpha(p_k, t^{up}_{k-1})  t^{up}_{k-1}\\\\
	\partial' q'^{down} _{k}&=& \sum \alpha (q'^{down} _{k}, q'^{down} _{k-1}) q'^{down} _{k-1}  \\ &+& \sum \alpha (q'^{down} _{k}, t'^{down} _{k-1}) t'^{down} _{k-1}  \\ &+&  \sum \alpha (q'^{down} _{k}, p_{k-1}) p_{k-1} \\\\ \partial'q^{up} _{k}&=& \sum \alpha (q^{up} _{k}, q^{up} _{k-1})  q^{up} _{k-1} \\ &+&  \sum \alpha (q^{up} _{k}, t^{up} _{k-1})  t^{up} _{k-1}  \\\\ \partial' t'^{down} _{k}&=& \sum \alpha (t'^{down} _{k}, t'^{down} _{k-1}) t'^{down} _{k-1}  \\ &+&\sum \alpha (t'^{down} _{k}, q'^{down} _{k-1}) q'^{down} _{k-1} \\ &+&  \sum \alpha (t'^{down} _{k}, p_{k-1}) p_{k-1}   \\\\ \partial't^{up} _{k}&=& \sum \alpha (t^{up} _{k}, t^{up} _{k-1})  t^{up} _{k-1} \\ &+&  \sum \alpha (t^{up} _{k}, q^{up} _{k-1})  q^{up} _{k-1}
       \end{eqnarray*}
     \end{itemize}

These sums extend over all the elements on the right hand
side and $\alpha $ is the number of gradient flow lines (mode 2) between the corresponding critical points.
     We want to prove ${\partial'}^ 2 = 0$ over $C_k(Y)$ by equating
     $\partial'$ with the boundary operator $\partial^M$ of  Floer  theory
     which is of the form
$\partial^M (s_k) = \sum \alpha (s_k, s_{k-1})  s_{k-1} $ for a gradient
vector field and counts the number of gradient flow lines $\alpha$ (mod 2)
between two rest points with relative index difference one, without any
partitioning on the set of isolated rest points $s_k$ of index $k$. In our case,
we have such a partitioning and therefore more refined relationships in the
definition of $\partial'$.  And then,  for all the generators
of $C_*(Y)$,  $\partial' = \partial^M$;  we show this equality
for $p_k,q^{up} _{k}, t'^{down} _{k}  $ as for the other cases it can be
similarly proved. \\
If we consider such a partitioning on the set of rest points of our vector field we have:  \\ 
      \begin{eqnarray*}
	\partial^M (p_k) &=& \sum \alpha (p_k, p_{k-1})  p_{k-1} +  \sum
        \alpha (p_k, q^{up}_{k-1})  q^{up}_{k-1}   \\&+&  \sum \alpha (p_k,
        q'^{down}_{k-1})  q'^{down}_{k-1} + \sum \alpha (p_k, t^{up}_{k-1})
        . t^{up}_{k-1}   \\&+&  \sum \alpha (p_k, t'^{down}_{k-1})  t'^{down} 
          \end{eqnarray*}
          \\
        Comparing this formula with that of $\partial' p_k $  we see that we
        have two extra terms in the latter; as we have explained after the
        definition of $\partial$,  the 3th
        and the 5th term are not present in the former case.  
        To have  $ \partial^M(q^{up} _{k}) = \partial'(q^{up} _{k})$,  the
         three coefficients 
	$$\alpha (q^{up} _{k}, q'^{down} _{k-1}), \alpha (q^{up} _{k},
        t'^{down} _{k-1}), \alpha (q^{up} _{k}, p_{k-1})$$ need to be zero. The
        first one is zero since there are exactly two gradient flow lines
        (heteroclinic orbits) from  $q^{up} _{k}$ to  $q'^{down} _{k-1}$
        which correspond to replacement of an orbit $O_{k-1}$. We note that
        for the other $q'^{down} _{k-1}$  coming from other orbits $\alpha$ is
        zero by definition of $\partial$ over $C_k(X)$ as otherwise in $X$ we
        would have flow lines between two orbits of the same index which is
        not possible by the simple generalised Morse-Smale condition. For the same reason, the
        second element is also zero since there is no flow line from
        $q^{up}_k$ to $t'^{down} _{k-1}$. Also the last $ \alpha$ is zero as
        otherwise there would be flow lines from a periodic orbit of index
        $k-1$ to an isolated rest point with index $k-1$ in $X$, again
        violating simple generalised Morse-Smale. \\
        Finally $ \partial^M(t'^{down} _{k}) =   \partial'(t'^{down} _{k})$ if
        we show that $ \alpha (t'^{down} _{k}, t^{up} _{k-1}) $ and $ \alpha
        (t'^{down}, q^{up} _{k-1})$ are zero. If not,  there would be two
        orbits  in $X$ with index difference two which are the boundaries of a
        cylinder, which is not possible.    \\
        Therefore over $C_*(Y)$,  $\partial^M=  \partial' $ and hence
        $\partial'^ 2 = 0$ by classical Morse-Floer theory. 
      \item[3.] We now define  $\varphi_* : C_*(X) \longrightarrow C_*(Y)$. 
        For $ 0\leq k \leq m$, we put
        \begin{eqnarray*}
	\varphi_*( p_k)= p_k,\quad \varphi_*  (O^0_ {k}) = q'^{down}_k ,\quad \varphi_*( O^1_ {k-1})= q^{up}_k ,  \\ \quad \varphi_*( H_ {k}^0)= t'^{down}_k  ,\quad  \varphi_*( H_ {k-1}^1)= t^{up}_k. 
 \end{eqnarray*}
	 $\varphi_* $ is an isomorphism by the above construction of the
         rest points of $Y$. To prove $\varphi_*$ is a chain map  from
         $C_*(X) $ to $ C_*(Y)$,  we should have $ \partial' \varphi_* =
         \varphi_*   \partial  $. 
Here, we show this equality  for one of the generators of $C_k(X)$ and for the
others it can be similarly obtained. For $O^1_ {k-1}$ we have:  
\begin{equation*}
\begin{aligned}[t]
 \varphi_* \partial ( O^1_ {k-1}) 
    &=  \varphi_*   \left(\sum \alpha( O_{k-1},  O_{k-2}) . O^1_{k-2} + \sum \alpha ( O_{k-1}, H_ {k-2}) . H_ {k-2}^1 \right)\\
    &= \sum \alpha (q^{up} _{k}, q^{up} _{k-1}) . q^{up} _{k-1} +  \sum \alpha (q^{up} _{k}, t^{up} _{k-1}) . t^{up} _{k-1}\\
    &= \partial'q^{up} _{k}\\
    &= \partial' \varphi_* ( O^1_ {k-1}) 
\end{aligned} 
\end{equation*}
Therefore  $\partial' \varphi_* = \varphi_*   \partial $ and  since ${\partial'}^2 = 0$  and  $\partial^ 
2 =\varphi_*^{-1} {\partial'}^ 2 \varphi_*, \partial^ 2 = 0$
  \end{proof} 
  \medskip
We can then define $\mathbb{Z}_2$ Morse-Floer homology of $M$ by putting 
for each $k, 0 \leq k \leq m $, 
\begin{center}
$H_k (M, \mathbb{Z}_2) = \frac {ker (\partial_k) }{image (\partial_{k+1})}$.\\ 
\end{center}

\begin{remark}
   Although here we do not treat orientations, we observe from  the following figure that in the above equalities $ \varphi_*$ preserves the parity of  $\alpha$ as each connected component of $M(O_{k-1},  O_{k-2})$  corresponds to exactly one gradient flow line from  $q^{up} _{k}$ to $ q^{up} _{k-1}$ (and exactly one flow line from  $q'^{down} _{k-1}$ to $ q'^{down} _{k-2}$). Similarly the same happens when  we consider connected components of $M(O_{k-1},  H_{k-2})$. 
  
\begin{center}
\includegraphics[width=5 cm]{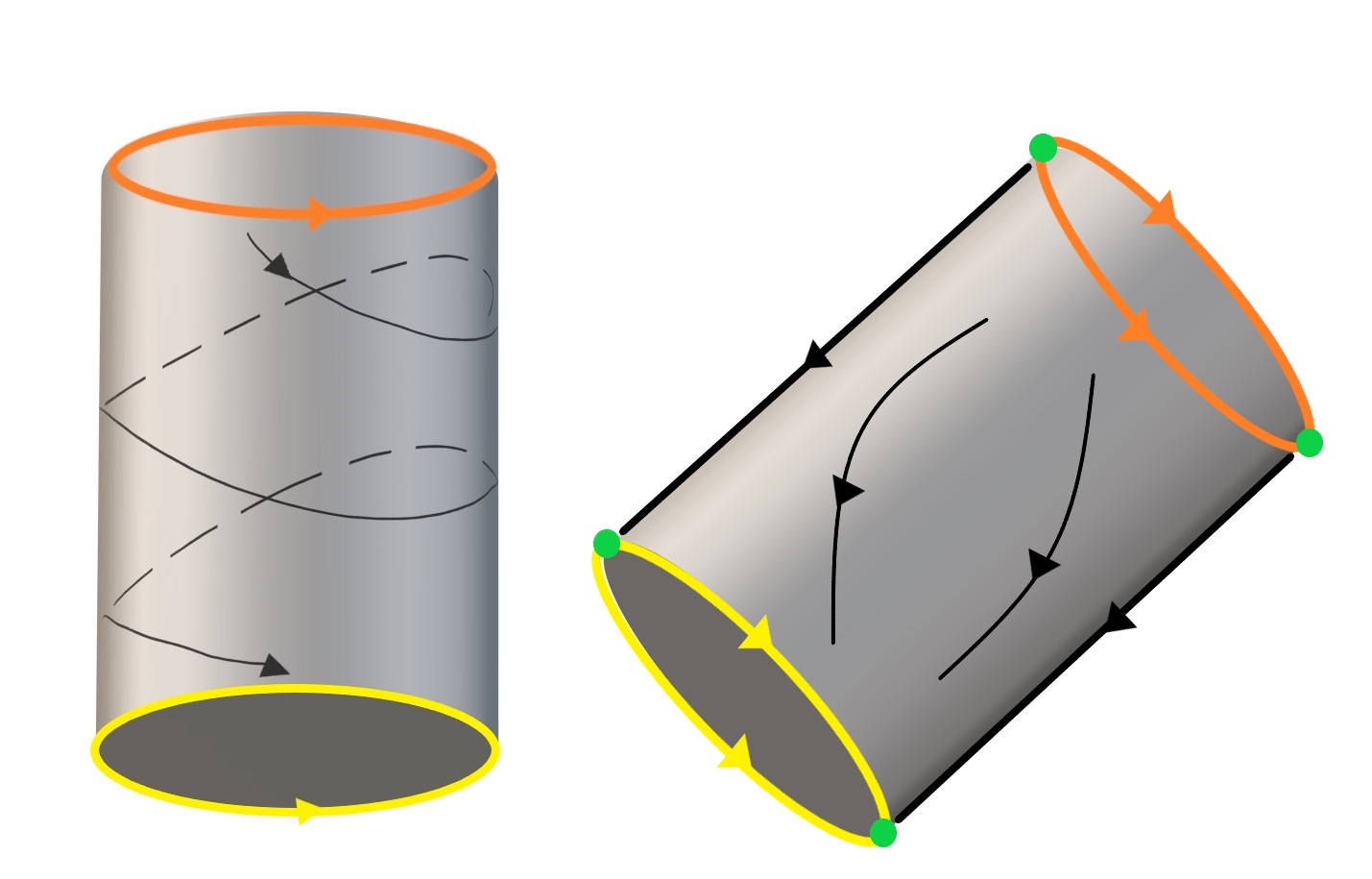} 
\end{center}
\medskip
\end{remark} 
  \begin{remark}
      As pointed out by the authors of \cite{Clemens}, the above boundary operator does not necessarily square to 0  for generalised Morse-Smale systems. However many typical examples including the ones presented in \cite{Clemens} are of  simple type (see definition \ref{3}).  In such cases after simplifying the system using the  Morse cancellation  theorem and/or the reverse of Franks' perturbation procedure, our boundary operator does square to 0 and can therefore be used to obtain the Morse-Floer homology. \end{remark} 
 
 \subsection{Computing Homology Groups of Smooth Manifolds}
We shall now illustrate the simple computation of Floer homology for some smooth vector fields.

1. Let the sphere $S^2$ be equipped with a vector field $V$ which has two isolated rest points of index zero at the north (N) and the south (S) pole, and one periodic orbit O of index one on the equator. Then \\ 
 \begin{wrapfigure}{r}{0.4\textwidth}
\hspace{1.5 cm}\vspace{0.4 cm} \includegraphics[width=3.5 cm]{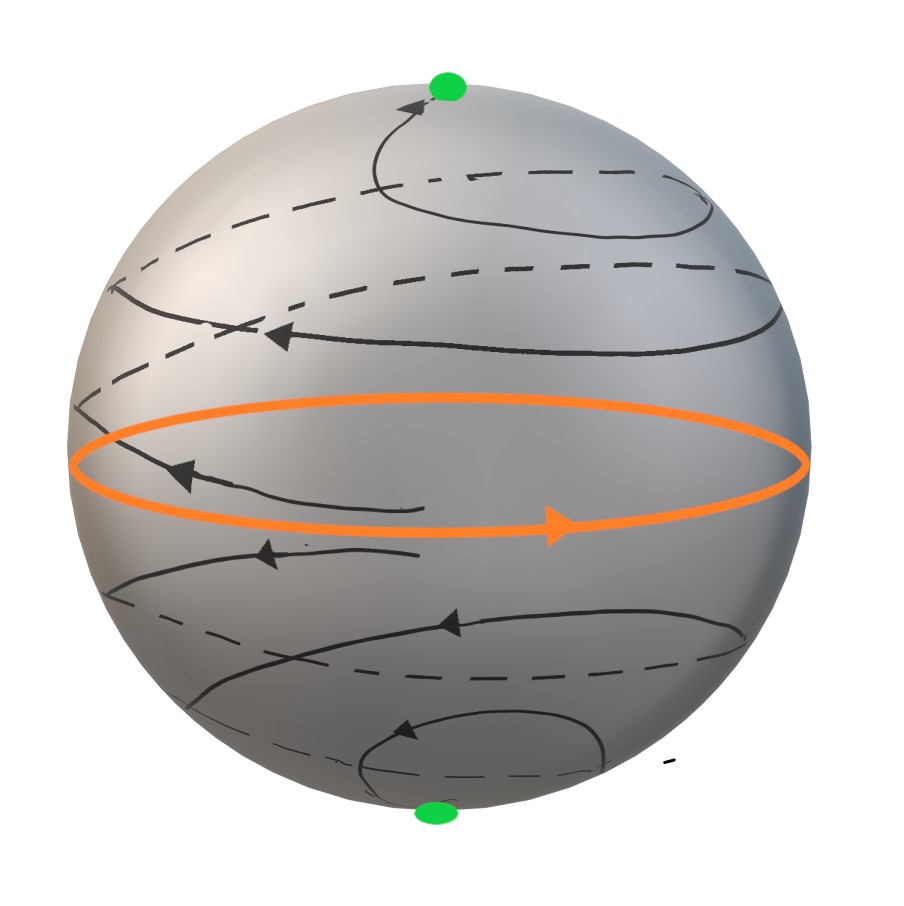} 
\end{wrapfigure}
\begin{eqnarray*}
  C_2&=& \left(  O^1_{1}\right) \\
  C_1&=& \left(  O^0_ {1}
         \right) \\
  C_0&=& \left( N_0, S_0 \right) 
  \end{eqnarray*}

$\partial_2 O^1_{1} =0$  since there is no closed orbit of index $0$ and therefore $O^1_{1}$ is the only generator for $H_2(M, \mathbb{Z}_2)$.\\
$\partial_1 O^0_{1}=  \alpha ( O_{1}, N_{0}) .N_{0} + \alpha ( O_{1}, S_{0}) .S_{0} = N_0+S_0 \neq 0 $ and therefore  $O^0_{1}$ does not contribute to $H_1(M, \mathbb{Z}_2)$ and $H_1(M, \mathbb{Z}_2)=0$.
 $\partial_0  N_0= 0=\partial_0  S_0$ but since $N_0+S_0$ is in the image of   $\partial_1$ therefore we have a single generator for  $H_0(M, \mathbb{Z}_2)$. 
\begin{wrapfigure}{r}{0.4\textwidth}
\hspace{1.5 cm}\vspace{0.4 cm} \includegraphics[width=3.5 cm]{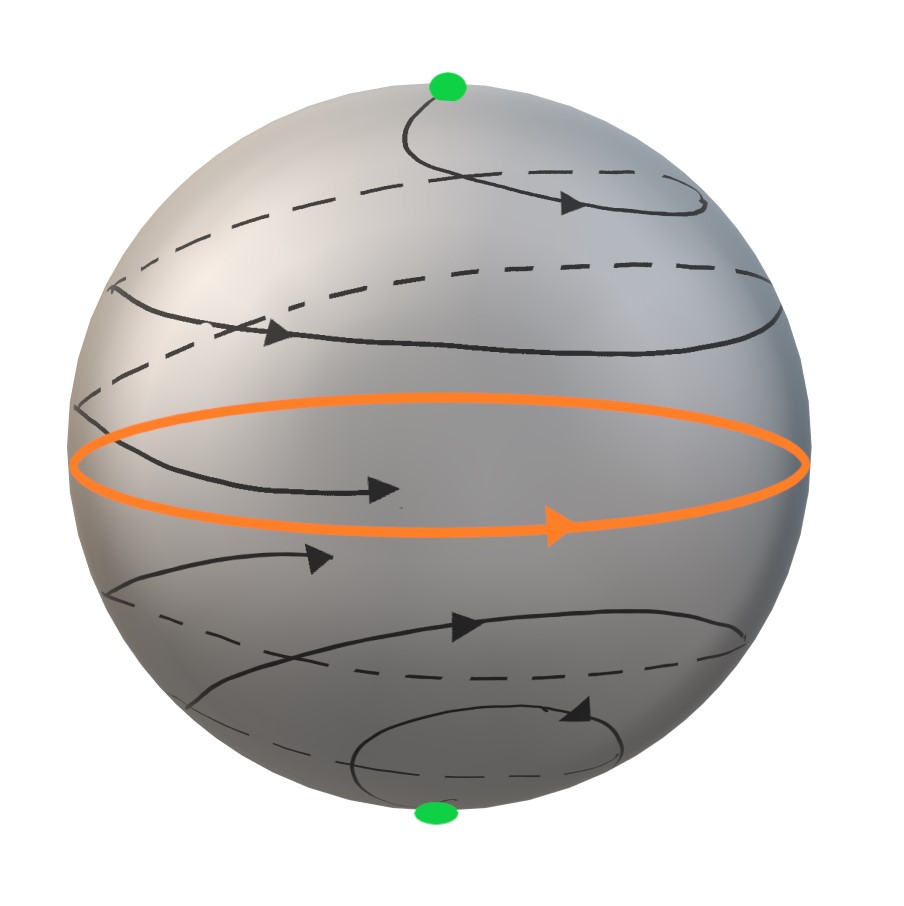} 
\end{wrapfigure}
\medskip

2. If we reverse the orientation of flow lines in the previous example, the isolated rest points at the north and south pole will get index two and the index of the  periodic orbit becomes zero. Therefore:  \\
\begin{eqnarray*}
  C_2&=& \left( N_2, S_2 \right) \\
  C_1&=& \left(
         O^1_{0}\right) \\
  C_0&=& \left(   O^0_ {0} \right)
  \end{eqnarray*}
$\partial_2  N_2= O^1_{0}=\partial_0  S_2$ and $N_2-S_2$ is the generator for $H_2(M, \mathbb{Z}_2)$.\\
Also $\partial_1 O^1_{1} =0$  but since $O^1_{1}$ is in the image of $\partial_2 $ it does not contribute to  $H_1(M, \mathbb{Z}_2)$.
Finally $\partial_0 O^0_{0}= 0$ and therefore  $O^0_{0}$ is the only generator for  $H_0(M, \mathbb{Z}_2)$. \\  \\
3. Consider $S^2$ with a vector field $V$ which has two isolated rest points,  at the north  pole of index zero and at the south pole of index two, one orange homoclinic orbit $H$ of index one and one yellow periodic orbit $O$ of index zero. \\
 \begin{wrapfigure}{r}{0.4\textwidth}
\hspace{1.5 cm}\vspace{0.4 cm} \includegraphics[width=3.3 cm]{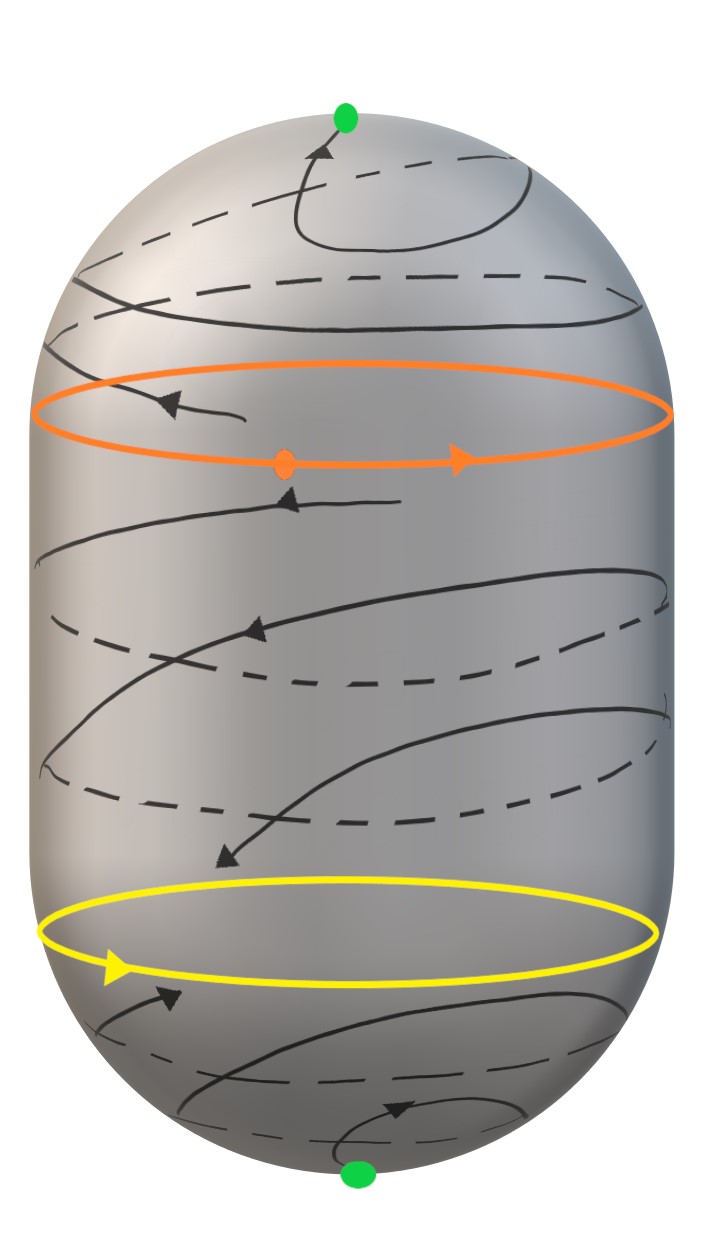} 
\end{wrapfigure}
\begin{eqnarray*}
  C_2&=& \left( S_2, H_{1}^1 \right) \\
  C_1&=& \left(  H_
         {1}^0, O^1_ {0}  \right) \\
  C_0&=& \left( O^0_ {0}, N_0 \right)
  \end{eqnarray*}

$\partial_2 S_2= O^1_ {0}= \partial_2 H_{1}^1 $ and therefore $S_2- H_{1}^1 $ is the only generator for $H_2(M, \mathbb{Z}_2)$.\\
$\partial_1 H_ {1}^0=  \alpha ( H_ 1, O_{0}) .O^0_ {0} + \alpha (
H_ 1, N_{0}) .N_{0} =O^0_{0} + N_{0} \neq 0 $\\
and therefore  $H_ {1}^0$ does not contribute to $H_1(M, \mathbb{Z}_2)$. On the other hand, $\partial_1 O^1_ {0}= 0 $ but since $O^1_ {0}$ is in the image of $\partial_2$ it does not contribute to $H_1(M, \mathbb{Z}_2)$ and therefore $H_1(M, \mathbb{Z}_2)=0 $. 
 $\partial_0  N_0= 0=\partial_0 O^0_ {0} $ but since $O^0_ {0} + N_{0}$ is in the image of $\partial_1$ therefore we have just one generator for  $H_0(M, \mathbb{Z}_2)$.   \\ \\ \\

 4. Finally, a two dimensional Torus $T^2$  with a vector field $V$  with two periodic orbits $O_1$  and $O'_0$:  

 \
 \begin{eqnarray*}
   C_2&=& \left(  O^1_{1}\right) \\
   C_1&=&   \left(  O^0_ {1}, O'^1_{0}  \right)\\
   C_0&=& \left( O'^0_ {0} \right)
   \end{eqnarray*}
     \begin{wrapfigure}{r}{0.4\textwidth}
\hspace{1.5 cm}\vspace{0.4 cm} \includegraphics[width=4.5 cm]{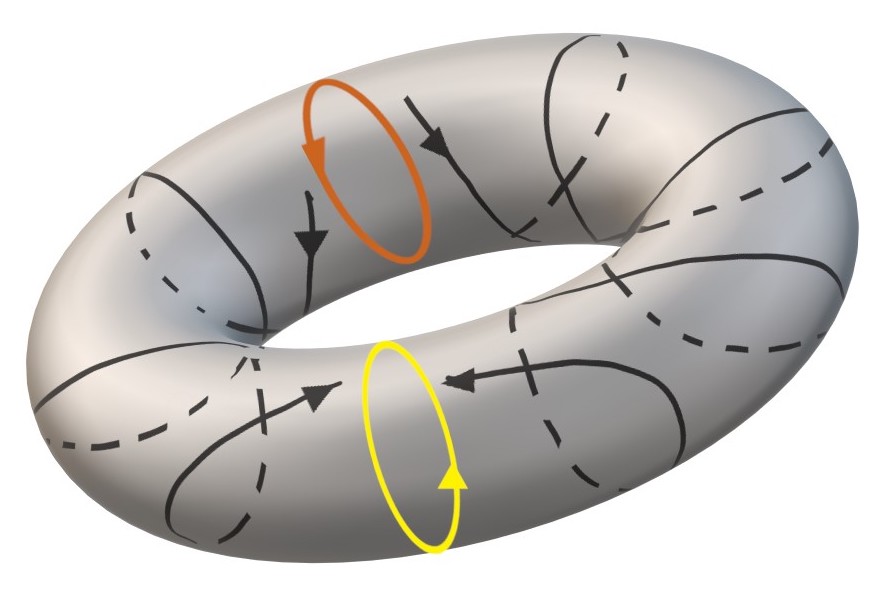} 
\end{wrapfigure} 
$\partial_2 O^1_{1}= 2. O'^1_{0}= 0$  therefore $O^1_{1}$ is the generator for $H_2(M, \mathbb{Z}_2)$.\\

$\partial_1 O^0_{1}= 2. O'^0_{0}= 0$ so  $O^0_{1}$ is a generator for  $H_1(M, Z)$. Also $\partial_1 O'^1_{0}= 0$ and therefore  $O'^1_{0}$ is another generator for $H_1(M, \mathbb{Z}_2)=0$.\\
Finally $\partial_0  O'^0_ {0}= 0$ and therefore we have one generator for  $H_0(M, \mathbb{Z}_2)$. \\ \\

  \begin{remark}   
  For a computation of  the homology groups of the first two examples via Morse- Bott theory (after turning the periodic orbits into critical submanifolds of a gradient flow), see \cite{Banyaga}. 
   \end{remark}

\section{Combinatorial Vector Fields}
 \subsection{Preliminaries}
 Forman introduced the notion of a combinatorial dynamical system on CW
 complexes \cite{Forman2}. He developed discrete Morse theory for the gradient
 vector field of a combinatorial Morse function and studied the homological
 properties of its dynamic \cite{Forman} . For the general combinatorial
 vector fields where as opposed to gradient vector fields, the chain recurrent set might also include closed
 paths, he studied some homological properties  by
 generalizing the combinatorial Morse inequalities. It remains, however, to
 construct a Floer type boundary operator for these general combinatorial
 vector fields.  We define a Morse-Floer boundary operator
 for combinatorial vector fields on a finite simplicial complex. With this
 tool  we no longer need a Morse function to compute the Betti numbers of the complex.      
  Combinatorial vector fields can be considered as the combinatorial version
  of smooth Morse-Smale dynamical systems on finite dimensional manifolds;
  here in contrast to the smooth case we cannot have homoclinic points and
  homoclinic orbits as here, we cannot have  a continuous bifurcation between
  a pair of heteroclinic orbits and a closed one, and in particular none with
  a homoclinic orbit in the middle.\\
  We now recall some of the main definitions that Forman introduced. {Let $M$ be a finite CW complex of dimension $m$, with $K$ the set of open cells of $M$ and $K_p$ the set of cells of dimension $p$. If $\sigma$ and $\tau$ are two cells of $M$, we write $\sigma_p$ if $\dim(\sigma) = p$, and $ \sigma < \tau$  if $\sigma \subseteq \overline{\tau}$ where $\overline{\tau}$ is the closure of $\tau$ and we call $\sigma$ a face of $\tau$.} \\
 Suppose $\sigma_p$ is a face of $\tau_{p+1}$, $B$  a closed ball of dimension
 $p+1$ and $h: B\rightarrow M$  the characteristic map for $\tau$ i.e.,  a
 homeomorphism from the  interior of $B$ onto $\tau$.  
 \begin{definition}
  $\sigma_p$  is a \emph{regular face} of $\tau_{p+1}$ if 
   \begin{itemize}
       \item $h^{-1} (\sigma) \rightarrow \sigma $ is a homeomorphism.
     \item  $\overline{h^{-1} (\sigma)}$ is a closed $p$-ball.
    \end{itemize}
     \end{definition}
     
  Otherwise we say $\sigma $ is an \emph{irregular face} of $\tau$.  If $M$
    is a regular CW complex (such as a simplicial or a polyhedral complex) then all its faces are regular.

\begin{definition}
A \emph{combinatorial vector field} on $M$ is a map $V : K \rightarrow K \cup  {0}$
such that
    \begin{itemize}
    \item For each $p$, $V(K_p) \subseteq  K_{p+1} \cup  {0}$. 
     \item For each $\sigma_p \in K_p $, either  $V(\sigma)= 0 $ or  $\sigma$ is a regular face of  $V(\sigma)$.
 \item If $\sigma \in \mathrm{Image}(V)$ then $V(\sigma) = 0$.
  \item For each $\sigma_p \in K_p $ \\
  $\sharp \lbrace u_{p-1} \in K_{p-1} \mid V(u)= \sigma \rbrace \le 1$.  
    \end{itemize}
    \end{definition}
To present the vector field on $M$ for any $\sigma \in K$ where $V(\sigma) \neq 0 $  we usually draw an arrow on $M$ whose tail begins at $\sigma$ and  extend this arrow  into $V (\sigma)$. Thus, for each cell $\sigma_p$, there are precisely 3 disjoint possibilities:
    \begin{itemize}
    \item $ \sigma$ is the head of an arrow $(\sigma \in Image(V))$.
     \item  $ \sigma$ is the tail of an arrow ($V(\sigma) \neq 0)$.
 \item $ \sigma$ is neither the head nor the tail of any arrow $(V (\sigma) = 0$ and $\sigma \not\in \mathrm{Image}(V ) $; 
     \end{itemize}
 In the last case we call such a $\sigma_p$ a zero or rest point of $V$ of index $p$.     
Cells which are not rest points occur in pairs ($\sigma,V(\sigma)$) with $ \dim V (\sigma) = \dim \sigma + 1$. 
From now on and for simplicity we restrict ourselves to the special case of simplicial complexes, instead of CW complexes.
 As the combinatorial version of closed periodic orbits in smooth manifolds we have the next definition: 
\begin{definition}
Define a $V$-\emph{path} of index $p$ to be a sequence \\  
$ \gamma : \sigma_p^0,\tau^0_{p+1},\sigma_p^1,\tau_{p+1}^1,...,\tau^{r-1}_{p+1},\sigma^{r}_p $

such that for all $ i = 1, . . . , r-1: $ \\
 $ V(\sigma^i)= \tau^i $ and $\sigma^i \neq \sigma^{i+1} < \tau^i$. \\
A \emph{closed path} $ \gamma$ of length $r$ is  a $V$-path such 
that $\sigma_p^0=\sigma_p^r$. Also $ \gamma$  is called  \emph{non-stationary} if $r > 0$.
\end{definition}
Forman showed that there is an equivalence relation on the set of closed paths by considering 
two paths $\gamma$ and $ \gamma'$ to be  equivalent if $ \gamma$ is the result of varying the starting point of $\gamma'$. An equivalence class of closed paths of index $k$ will be called a \emph{closed orbit} of index $k$ and denoted by $O_k$.
\begin{definition}
 We call an orbit $O_p$ non-twisted (or simple) if in its corresponding closed path \\
 $ \gamma : \sigma_p^0,\tau^0_{p+1},\sigma_p^1,\tau_{p+1}^1,...,\tau^{r-1}_{p+1},\sigma^{r}_p =\sigma_p^0 $, \\ 
 there exist an assignment of orientations on its $p+1$-simplexes such that the induced orientation of every two $p+1$-simplex on their common $p$-face cancel each other (namely they are opposite). Otherwise we call it twisted.  A discrete cylinder and torus have respectively simple $0$-orbit(s) and $1$-orbit(s) and discrete Mobius strip and  Klein bottle  have respectively a $0$-twisting orbit and a $1$-twisting orbit.
\end{definition}

\begin{definition}
The \emph{combinatorial chain recurrent set} $R(V)$ for a combinatorial vector field $V$  on $M$ is defined to be the set of simplices $\sigma_p$ which are either rest points of $V$ or are contained in some non-stationary closed path $\gamma$ ($ \gamma$
must have index either $p-1$ or $p$).

   \end{definition}
   
The chain recurrent set can be decomposed into a disjoint union of basic sets
$R(M)=\cup_{i} \Lambda _i$
where two simplices $ \sigma, \tau \in R(V)$ belong to the same basic set if
and only if there is a closed non-trivial $V$-path $ \gamma$ which contains
both $\sigma$ and $\tau$. Forman proved that if there are no non-stationary
closed paths, then $V$ is the combinatorial negative gradient vector field of
a combinatorial Morse function. However when $V$ has closed paths, then it
cannot be the gradient of a function. Subsequently he defined a combinatorial
''Morse-type'' function on $K$, called a  Lyapunov function, which is constant on each basic set, and has the property that, away from the chain recurrent set, $V$ is the negative gradient of $f$. 
\begin{remark}
This Lyapunov function can be considered as the combinatorial analogue of the Morse-Bott energy function which Mayer defined for Morse-Smale dynamical systems.
\end{remark}

\subsection{The chain complex of combinatorial vector fields}
Forman obtained Morse-type inequalities based on the basic sets of $V$ and showed that these sets  control the topology of $M$ \cite{Forman2}. 
In this section, we present  a direct way of recovering the homology of the underlying complex from the chain recurrent set of a combinatorial vector field on $M$ from our Floer type boundary operator;{  our main restriction is that the chain recurrent set should not have twisted orbits}. Our operator acts on chain groups generated by the basic sets and counts the number of suitable V-paths between elements of the chain recurrent set.
We consider $V$ to be a combinatorial vector field on a finite simplicial
complex $M$ { where $R(V)$ does not include twisted orbits}. \\
We define the Morse-Floer complex of $V$ denoted by $(M, C_*(V), \partial)$ as follows. Let $C_k$ denote the finite vector space (with coefficients in $\mathbb{Z}_2$) generated by the  set of rest points $p_k $ and closed orbits $O_{k-1}$ of the vector field:  \\ 
	\[ \left(  p_k,  O^1_{k-1},  O^0_ {k}  \right) \]
in which by $O^1_{k-1}$ we mean the whole closed orbit $O_{k-1}$ of index
${k-1}$ and by $O^0_ {k}$ we mean an arbitrary simplex with dimension $k$ in
the closed orbit $O_k$. Similar to the smooth case, each such orbit carries
topology in two adjacent dimensions, namely a closed orbit $O_k$ generates an
element $O^1_{k}$ in $C_{k+1}$ and an element $O^0_{k}$ in $C_{k}$. We note
  that here, by definition of combinatorial vector fields, we do not have any
  $V$-path between the elements of the same $C_k$; but in order to get a Floer
  type boundary operator in the same way as  in the smooth setting we have to exclude three different cases in our vector field; we assume:
 \begin{enumerate}
 \item There is no $V$-path from a face of a critical simplex $p_k$ to a $(k-1)$-dimensional simplex 
    in an orbit $O_{k-1}$.
 \item  There is no $V$-path from a face of a $k$-dimensional simplex of an orbit of index $k-1$ to a critical simplex of dimension $k-1$.
  \item There is no $V$-path from a face of a $k$-dimensional simplex of an
    orbit 
    of index $k-1$ to a $(k-1)$-dimensional simplex of another
    orbit 
    of index $k-1$.
 \end{enumerate}
   This will be used in the proof of Thm. \ref{2}. In the smooth setting, the
   excluded cases cannot occur because of the  Morse-Smale transversality condition  and the simpleness of the system (defined in definition \ref{3}).  
 \\

 To be able to define the combinatorial Floer-type  boundary operator,  we
 have to transfer the idea of the number of connected components of the moduli
 spaces of flow lines to our combinatorial setting. As we saw, the number of
 these components (mod 2) plays a key rule in the definition of the boundary
 operator in the smooth setting. In the sequel, for two simplices of the
 same dimension $q$ and $q'$,  by $ q \perp q'$, we mean that $q$ and $q'$ are
 lower adjacent, i.e.,  they have a common face.  \\
 We have $V$-paths between closed orbits and rest points which make different following cases: \\
For two orbits $O_ {k-1}$ and $ O_ {k-2}$ we define the set $VP(O_ {k-1} , O_
{k-2})$  as the set of all V-paths starting from the faces of $(k-1)$- and
$k$-dimensional simplices of $O_ {k-1}$ and go to respectively $(k-2)$- and $(k-1)$-
dimensional simplices of $O_ {k-2}$. \\
If  $VP(O_ {k-1} , O_ {k-2})$  is non-empty, for $O^1_{k-1}$ and $O^1_{k-2}$, we define the higher dimensional spanned set of V-paths in  $VP(O_ {k-1} , O_ {k-2})$, denoted by $SVP(O^1_ {k-1} , O^1_ {k-2})$ to be   
 $$ \lbrace q \in K_{k},  q \in Image (V) \mid
 \exists \gamma \in VP(O_ {k-1} , O_ {k-2}), q \in \gamma   \rbrace.$$
 On this set we can then define a relation as follows. 
We say $q$ and $q'$ in $SVP(O^1_ {k-1} , O^1_ {k-2})$ are related ($q\sim q'$) if $q$ and $q'$ belong respectively to two V-paths $ \gamma : \alpha_{k-1}^0,...,q_k,...,\alpha^{r}_{k-1} $ and $\gamma' : \beta_{k-1}^0,...,q'_k,...,\beta^{s}_{k-1}$ where $\alpha_{k-1}^0$ and $\beta_{k-1}^0$ are faces of $k$-dimensional simplices of $O_ {k-1}$ and $\alpha_{k-1}^r$ and $\beta_{k-1}^s$ are some $k-1$  dimensional simplices in $O_ {k-2}$ such that one of the following situations happens:    
 \begin{itemize}
 \item Either $\alpha_{k-1}^0$ and $\beta_{k-1}^0$ coincide (and therefore $ \gamma$ and $ \gamma'$ are the same)  or
 \item $ \alpha_{k-1}^0\perp \beta_{k-1}^0 $ or
 \item  There is a sequences of  $k-1$ dimensional simplices $\theta_{k-1}^0,... \theta_{k-1}^z$, where  $\theta_{k-1}^0,... \theta_{k-1}^z$  are the faces of $k$ dimensional simplices in $O_ {k-1}$ such that $\alpha_{k-1}^0 \perp \theta_{k-1}^0$, $\beta_{k-1}^0 \perp \theta_{k-1}^z$ and  for each $i$, $\theta_{k-1}^i \perp \theta_{k-1}^ {i+1} $ and $\theta_{k-1}^i $ is the starting simplex of some $ \gamma \in VP(O_ {k-1} , O_ {k-2})$. 
 \end{itemize}
  It is straightforward to check that $\sim$ is an equivalence relation on $SVP(O^1_ {k-1} , O^1_ {k-2}) $. \\ 
On the other side for two arbitrary simplices of dimension $k-1$ and $k-2$ in respectively $O_ {k-1}$ and $O_ {k-2}$ we consider the following  equivalence relation( $\sim'$) on  $SVP(O^0_ {k-1}, O^0_ {k-2})$ which is defined as follows: 
\begin{equation*}
 \lbrace q \in K_{k-1}, q \in \small{Image (V)}  \mid  \exists \gamma \in VP(O_ {k-1} , O_ {k-2}),  q \in 
 \gamma \rbrace.
\end{equation*}
 We say $q$ and $q'$ in $SVP(O^0_ {k-1} , O^0_ {k-2})$ are related ($ q \sim' q'$) if there are $w$ and $w'$ in $SVP(O^1_ {k-1} , O^1_ {k-2})$ such that $q<w$ and $q'<w'$ and $w\sim w' $. By definition $\sim'$ is also an equivalence relation on $SVP(O^0_ {k-1} , O^0_ {k-2})$ and the number of its equivalence classes is exactly the number of equivalence classes of $\sim$ over $SVP(O^1_ {k-1} , O^1_ {k-2})$. \\ 
For instance consider the following triangulation of the torus which has two
closed orbits of index one and index zero, respectively shown by green and red
arrows. Here, $SVP(O^1_ {k-1} , O^1_ {k-2})$ is the set of all the two
dimensional coloured simplices and based on the above equivalence relation, this set is partitioned into two sets of yellow and pink two dimensional simplices. Also  $SVP(O^0_ {k-1} , O^0_ {k-2})$ is the set of all marked (with cross sign) edges which is partitioned into two sets, represented by orange and purple signs.

 \begin{center}
\includegraphics[width=4 cm]{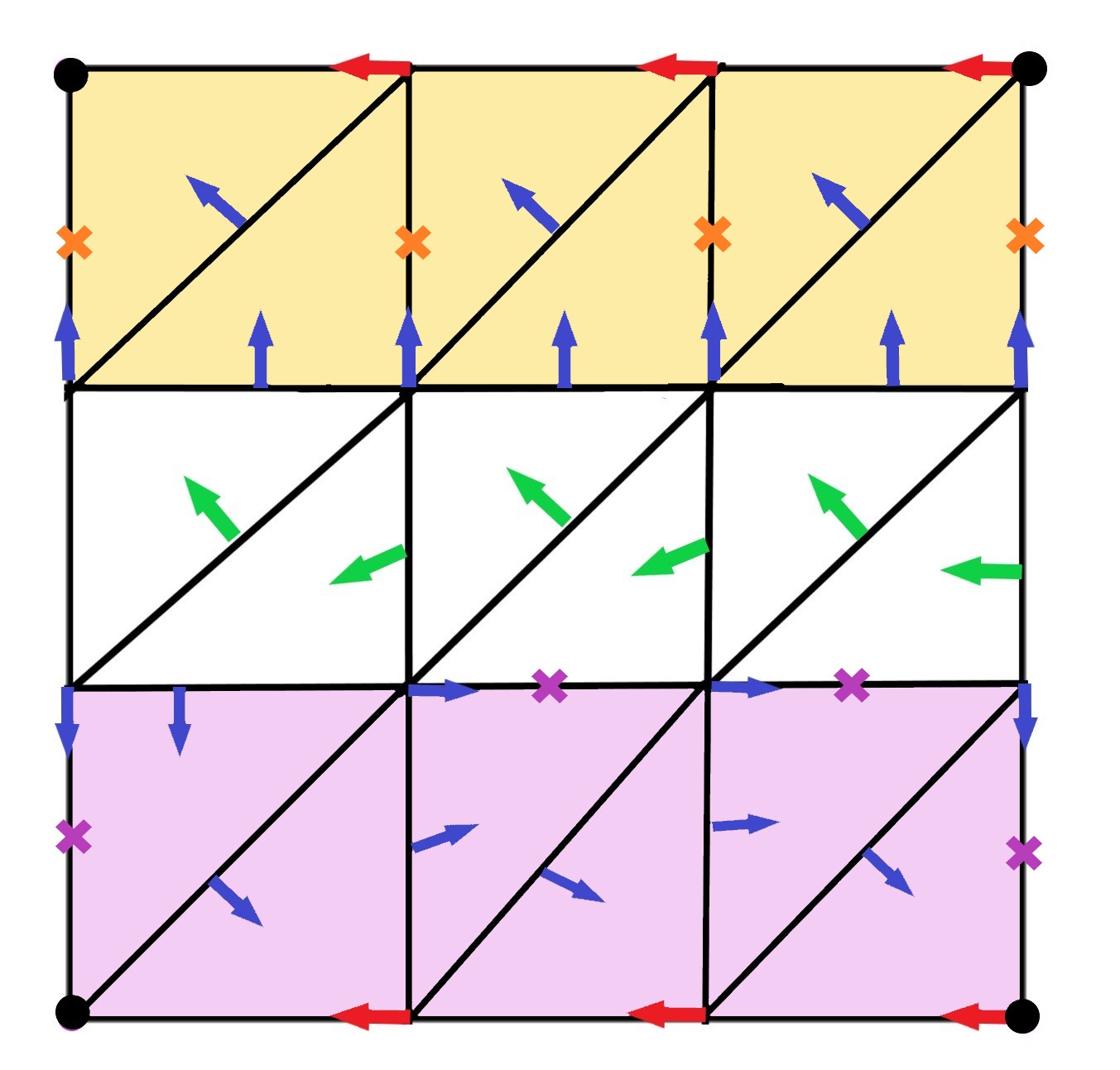} 
\end{center}
 
If for two orbits, $O_ {k-1}$ and $O_ {k-2}$,  $VP(O_ {k-1} , O_ {k-2})$ is empty and some of the faces of $O_ {k-1}$ (faces of both $k-1$ and $k$-dimensional simplices) coincide with $k-2$ and $k-1$ dimensional simplices in $O_ {k-2}$, we say $O_ {k-2}$ is attached to $O_ {k-1}$. 
In the tetrahedron shown below the bottom faces of the closed red orbit of
index one,  is the closed orbit of index zero with purple arrows:
 
 \begin{center}
\includegraphics[width=4 cm]{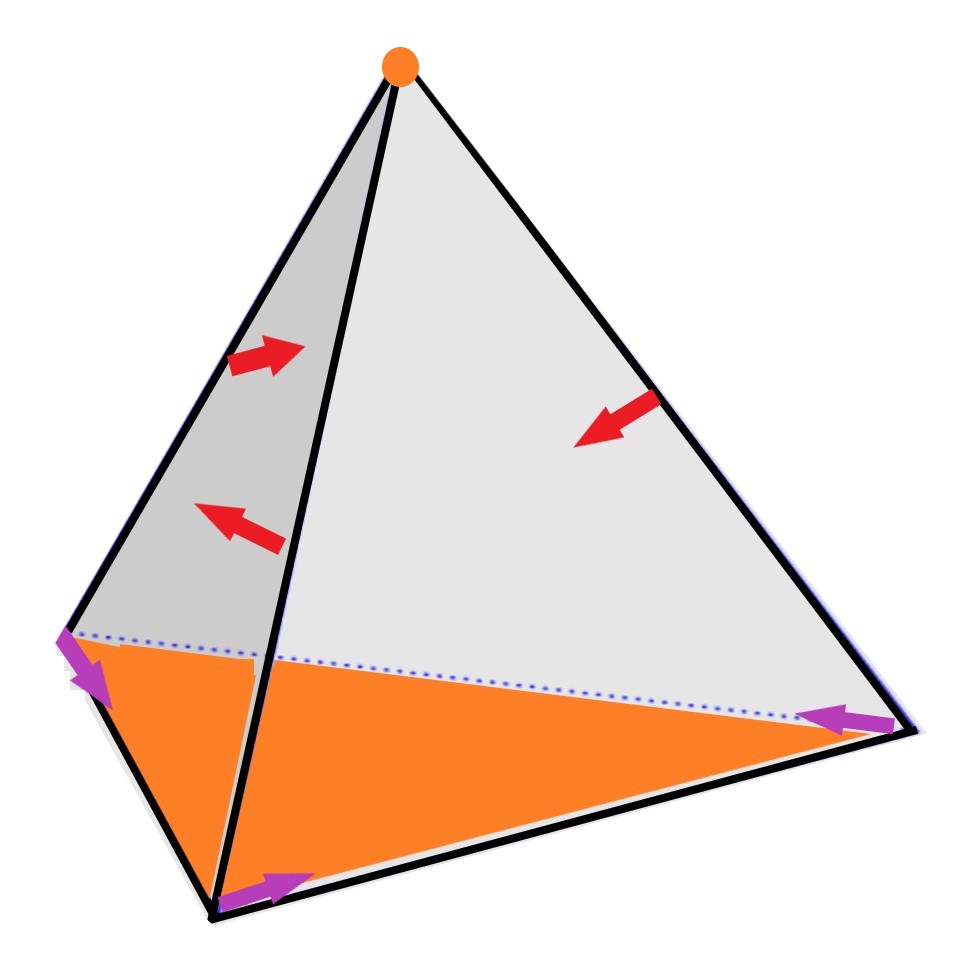} 
\end{center}

Also we could have V-paths, $VP(p_ k , O_{k-2})$,  from the faces of a
critical simplex $p_k$ of index k which go to the $k-1$ dimensional simplices
of some orbit of index $k-2$ ; we define the span set of these V-paths,
denoted by $SVP(p_ {k} , O^1_ {k-2})$ to be
$$SVP(p_ {k} , O^1_ {k-2}): = \lbrace q \in K_{k}, q \in Image (V)  \mid
\exists \gamma \in VP(p_ {k} , O_ {k-2}), q \in \gamma   \rbrace.$$
As above we can define an equivalence relation on this set in which the equivalence classes are obtained based on the following relation:  \\ 
$q\sim q'$ if they belong respectively to two V-paths $ \gamma : \alpha_{k-1}^0,...,q_k,...,\alpha^{r}_{k-1} $ and $\gamma' : \beta_{k-1}^0,...,q'_k,...,\beta^{s}_{k-1}$ such that either $\alpha_{k-1}^0$ and $\beta_{k-1}^0$ coincide or  $ \alpha_{k-1}^0\perp \beta_{k-1}^0 $ or there is a sequence of  $k-1$ dimensional simplices $\theta_{k-1}^0,... \theta_{k-1}^z$, where  $\theta_{k-1}^0,... \theta_{k-1}^z$  are the faces of $p_k$ such that $\alpha_{k-1}^0 \perp \theta_{k-1}^0$, $\beta_{k-1}^0 \perp \theta_{k-1}^z$ and  for each $i$, $\theta_{k-1}^i \perp \theta_{k-1}^ {i+1} $. We note that here, $\alpha_{k-1}^0$ and $\beta_{k-1}^0$ are faces of $p_k$ and $\alpha_{k-1}^r$ and $\beta_{k-1}^s$ are some $k-1$ elements of $O_ {k-2}$. \\

Also for $V$ on $M$,  for some rest point $p_k$ and some closed orbit
$O_{k-2}$,  the faces of $p_ {k}$ and $k-1$ dimensional simplices in $ O_
{k-2}$ might coincide; for instance in the above tetrahedron the faces of orange 2-d
rest simplex coincides with the one dimensional simplices in the closed orbit
of index zero with purple arrows. We consider this again as an \emph{attachment}. \\
 
In the third possible case, V-paths start from the faces of $k$-dimensional
simplices of a closed orbit of index $k$, $O_k$, and go to a rest simplex of
index $k-1$, $p_{k-1}$. We denote the set of such V-paths by  $VP(O_ {k} , p_
{k-1})$  and we consider $SVP(O^0_ {k} , p_ {k-1}): = \lbrace q \in K_{k}, q
\in Image (V)  \mid  \exists \gamma \in VP(O_ {k} , p_ {k-1}), q \in \gamma
\rbrace$. In this set we call two simplices $q$ and $q'$  equivalent if either
$q$ and $q'$ coincide or  $q \perp q'$ or we can find a sequence of simplices
in $SVP(O^0_ {k} , p_ {k-1})$  such as $\theta_{k}^0,... \theta_{k}^z$,
such that $q \perp \theta_{k}^0$, $q' \perp \theta_{k}^z$ and for each
$i$, $\theta_{k}^i \perp \theta_{k}^ {i+1} $. Here we have to exclude $p_
{k-1}$ for determining lower adjacency of $k$-dimensional simplices in
$SVP(O^0_ {k} , p_{k-1})$, namely if $ q \cap q' = p_ {k-1} $, they belong to
different classes.  As an example consider the following triangulation for the
torus with four orange rest simplices, one of index two, two of index one and another
one of index zero, and a closed red orbit of index one. Here, the edges marked
by cross signs are the edges in  $SVP(O^0_ {1} , p_ {0})$, which is portioned into two pink and yellow marked edges. 
 {If $VP(O_ {k} , p_{k-1})$ is empty, but $O_k$ and $p_ {k-1}$ have a non-empty intersection, we have another type of attachment. For an example of this case see the top critical vertex and the red $O_1$ in the above tetrahedron.
 }
 \begin{center}
\includegraphics[width=4 cm]{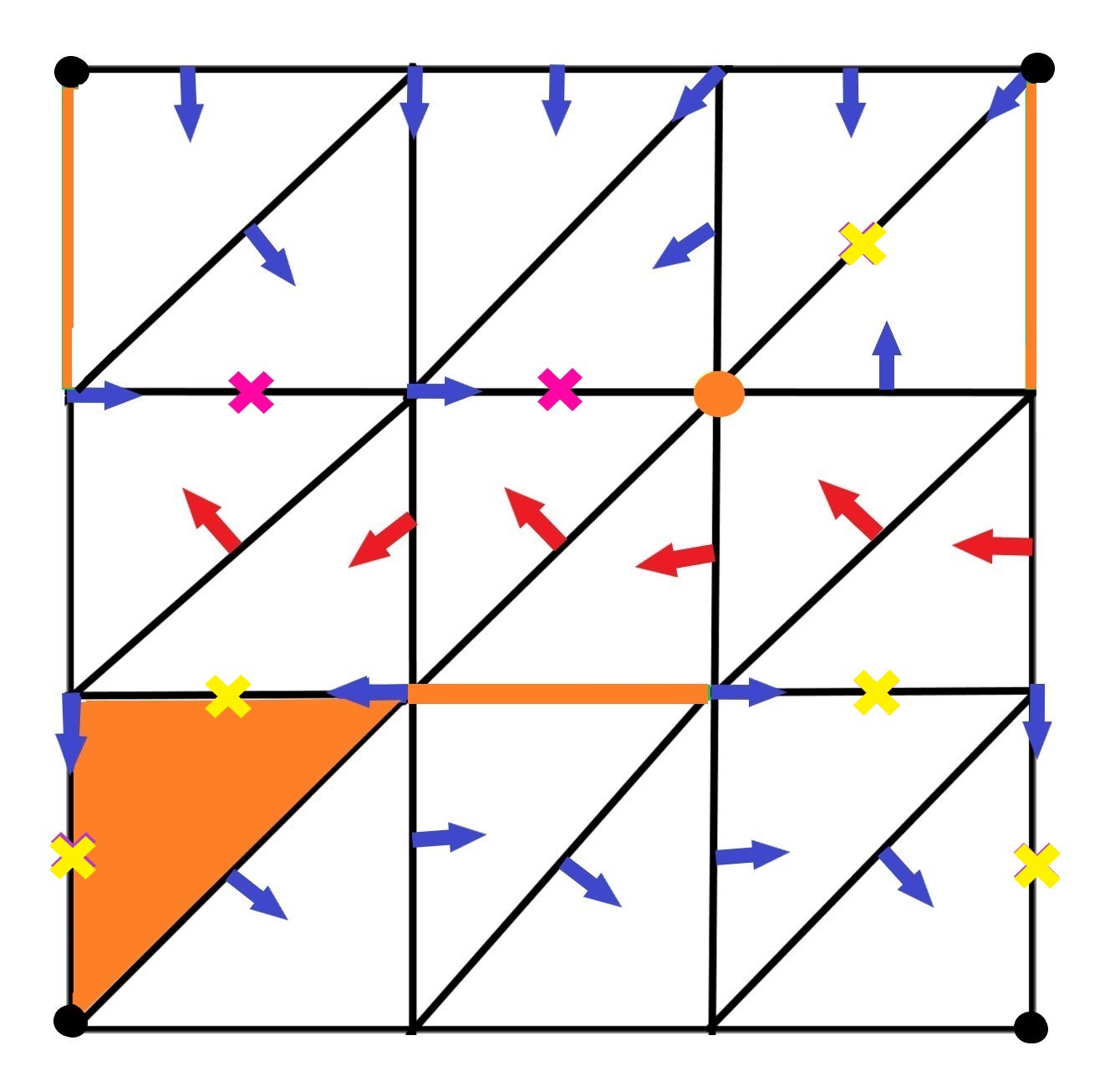} 
\end{center}
   
Finally if we substitute $O_ {k}(O^0_ {k})$ in the third case by a rest point of index k, $p_k$, we count the number of equivalence classes of $SVP(p_ {k} , p_ {k-1}): = \lbrace q \in K_{k}, q \in Image (V)  \mid  \exists \gamma \in VP(p_ {k} , p_ {k-1}), q \in \gamma \rbrace$ where  $VP(p_ {k} , p_ {k-1})$ is the set of all V-paths starting from the faces of $p_k$ and go to $p_ {k-1}$ by passing through $k$-dimensional simplices, based on the following relation:  \\  
   We say $q$ and $q'$ in $SVP(p_ {k} , p_ {k-1})$ are related ($q\sim q'$) if there is a V-path in $VP(p_ {k} , p_ {k-1})$ which includes both $q$ and $q'$. Therefore the number of equivalence classes here is  the number of V-paths starting from the faces of $p_k$ which go to  $p_ {k-1}$.\\  
The differential $\partial_k : C_k(V) \longrightarrow C_{k-1}(V)$  counts the number of the above equivalence classes mod 2, denoted by $\alpha$, and for the three types of attachments we consider $\alpha= 1$.   
That is,
\begin{eqnarray*}
\partial p_k &=& \sum \alpha (p_k, p_{k-1})  p_{k-1} + \sum  \alpha (p_k,  O^1_{k-2})   O^1_{k-2}\\ \partial  O^0_{k}&=& \sum \alpha ( O^0_{k},  O^0_{k-1})  O^0_{k-1}+  \sum \alpha ( O^0_{k}, p_{k-1}) p_{k-1} \\  \partial  O^1_{k-1}&=& \sum \alpha( O^1_{k-1},  O^1_{k-2})  O^1_{k-2}  
\end{eqnarray*}
where the sums extend over all the elements on the right hand side; for instance the second sum in the first line is over all closed orbits $O^1_{k-2}$ of index $k-2$. In Forman's discrete Morse theory where there is no closed orbit (and therefore the combinatorial vector field is the negative gradient of a discrete Morse function), $\alpha (p_k, p_{k-1})$  is the number of gradient V-paths from the faces of the rest point $p_k$ of higher dimension to the rest point of lower dimension $p_{k-1}$ (in this case all the coefficients in the above formula except the first coefficient in the first line are zero). \\

  \begin{theorem}\label{2}
  $\partial^ 2 = 0$.
   \end{theorem}
To prove this theorem similarly to Theorem \ref{1} in the smooth case, we
introduce a procedure to replace  any closed path of index $p$
(correspondingly its orbit of index $p$, $O_{p})$  with a rest  point of index
$p$ and one of index $p+1$ which are joined by two gradient V-paths starting
from the faces of a higher dimensional rest point and going to the lower dimensional rest point.  \\   	
We assume $V$ has a finite number of simple closed non-stationary paths (orbits) and rest simplices. Choose arbitrarily one of these closed paths $\gamma$ of index $p$,  $ \gamma : \sigma_p^0,\tau^0_{p+1},\sigma_p^1,\tau_{p+1}^1,...,\tau^{r-1}_{p+1},\sigma^{r}_p = \sigma_p^0 $. $\gamma$ is a sequence of $p$ and $(p+1)$ dimensional simplices. Take one of the $p+1$ dimensional simplices $ \tau^{k}_{p+1}$ where $ k\neq r-1$. (Note that for non-stationary closed paths such a  $\tau^{k}$ always exists). We consider the following two sets of the simplices of $\gamma$ by preserving the orders in each of the sets: \\ 
  \begin{center}
 $\sigma^{0}_p,...., \tau^{k-1}_{p+1}, \sigma^{k}_p,  \tau^{k}_{p+1} $  and \\ 
 $\tau^{k}_{p+1}, \sigma^{k+1}_p, \tau^{k+1}_{p+1}, ...., \sigma^{r}_p = \sigma_p^0 $   
  \end{center}
 where the union of the elements in these sets consists of  all the simplices of $\gamma$ and their intersection is the starting simplex of the closed path $ \sigma_p^0 $ and the one of higher dimension that we took $ \tau^{k}_{p+1}$.  

We keep the arrows in the second set as they are in $\gamma$ and in the first set we reverse the direction of V-path from $ \sigma_p^0 $ to $ \tau^{k}_{p+1}$.  Namely instead of a pair $(\sigma^{s}_p, \tau^{s}_{p+1}) $ in $\gamma$ (for $ 0\leq s \leq k$)  we will have $( \sigma^{s}_p, \tau^{s-1}_{p+1})$ in our vector field where the two simplices $ \sigma_p^0 $ and $\tau^{k}_{p+1} $ will no longer be the tail and head of any arrow; therefore by definition both of them become rest points  and there is no other rest point in $ \gamma$  created in this process. We note that in this procedure we just change the arrows in $O$ and the other pairs of the vector field (outside $O$) are not changed. 
        
    \begin{proof} of theorem \ref{2}.
     If by the help of above procedure we replace all the closed paths (orbits) by two rest points whose indices (dimensions) differ by one we get a  vector field $V'$ which has no closed path (orbit) and therefore there is a discrete Morse function on $M$ whose gradient is $V'$. $V'$ has all the rest  points of $V$ and two rest points $q^{up}_k$ (the simplex of higher index) and $q'^{down}_{k-1}$ (the simplex of lower index) instead of every orbit $O_{k-1}$ of index $k-1$. Then we have the following three steps to prove the theorem: 
  \begin{itemize}
  \item[1.] We consider $C_k(V')$  to be the finite vector space (with coefficients in $\mathbb{Z}_2$) generated by 
   	\[ \left(  p_k,  q^{up}_k,  q'^{down}_{k}  \right) \]
 where in this  set  $q^{up}_k$ comes from an orbit of index $k-1$ and  $q'^{down}_{k} $ comes from the replacement of an orbit of index $k$. 

    \item[2.] we  define a boundary operator $\partial'$ and consequently a chain complex corresponding to $(V',C_*(V'),  \partial')$. 
  \item[3.]  Then we prove there is an isomorphism (chain map) $\varphi * : C_*(V) \longrightarrow C_*(V')$ .
 \end{itemize}
    Since $\varphi *$ is an isomorphism we get our desired equality   $\partial^ 2 = 0$ as $\partial = \varphi ^{-1}  \partial'$ and  $\partial^ 2=\varphi ^ {-1}  {\partial'} ^ 2$.
 \begin{itemize}    
  \item[1.] All the
elements of the chain recurrent set of  $V'$ are rest simplices and for each
index $k$  they can be  partitioned into three different sets $p_k, q^{up}_k,
q'^{down}_k $. Here, in contrast  to the smooth case, orbits and rest points can have non-empty intersections; in particular for different types of attachments, the pairwise intersections are non-empty. However partitioning of rest simplices is possible since the indices of the rest simplices are the same as their dimensions and after converting orbits into two rest simplices, they will belong to different chain groups (in adjacent dimensions). 
    \item[2.] We define  $\partial' : C_k(V') \longrightarrow C_{k-1} (V')$ as
      follows: 
      \begin{eqnarray*}
	\partial' p_k &=& \sum \alpha (p_k, p_{k-1}) . p_{k-1} + \sum  \alpha(p_k, q^{up}_{k-1}) . q^{up}_{k-1} \\ \partial' q'^{down} _{k}&=& \sum \alpha (q'^{down} _{k}, q'^{down} _{k-1}) .q'^{down} _{k-1}  \\ &+&  \sum \alpha (q'^{down} _{k}, p_{k-1}) .p_{k-1} \\ \partial'q^{up} _{k}&=& \sum \alpha (q^{up} _{k}, q^{up} _{k-1}) . q^{up} _{k-1} \\ 
    \end{eqnarray*}
    \end{itemize}

	To prove $\partial'^ 2 = 0$ over $C_k(V')$, we want to equate
        $\partial'$ with the discrete Morse-Floer  boundary operator
        $\partial^M$  of a combinatorial gradient vector field of the form
        $\partial^M (s_k) = \sum \alpha (s_k, s_{k-1}) s_{k-1} $. There we
        count the number of gradient V-paths $\alpha$ (mod 2) between two rest
        points of relative index difference one without any such partitioning
        on the set of rest simplices $s_k$ of index $k$. In our case where we have such kind of partitioning we should show that for all the generators of $C_*(V')$,  $\partial' = \partial^M$.  \\
After the preceding procedure, we have: \\ 
   \begin{eqnarray*}
	\partial^M (p_k) &=& \sum
\alpha (p_k, p_{k-1}) . p_{k-1} +  \sum \alpha (p_k, q^{up}_{k-1})
. q^{up}_{k-1} \\ &+&  \sum \alpha (p_k, q'^{down}_{k-1}) . q'^{down}_{k-1} ;
    \end{eqnarray*} 
comparing this formula with that of  $\partial' p_k $ in the above formula we
see that there is one extra term in the latter; {because we exclude case (1) in our vector field, the third sum is not present in the former case. } \\
As in the previous discussion to have  $ \partial^M(q^{up} _{k}) = \partial'(q^{up} _{k})$, the following two coefficients should be zero: \\ 
	$\alpha (q^{up} _{k}, q'^{down} _{k-1}), \alpha (q^{up} _{k},
        p_{k-1})$. In the first case if $q^{up} _{k}$ and $q'^{down} _{k-1}$ are coming from replacement of the same orbit $O_{k-1}$, we will have exactly two $V'$-paths from the faces of $q^{up} _{k}$ to  $q'^{down} _{k-1}$ and it is zero mud 2. If they are not obtained from replacement of the same orbit $O_{k-1}$, $\alpha$ is zero as {  otherwise  in $V$ we would have $V$-paths between two orbits of the same index which either contradicts the non-existence of $V$-paths between elements of the same chain group or violates our exclusion (3) on the vector field. On the other hand, the second $ \alpha$ is zero as otherwise it violates our assumption (exclusion 2) on the vector field}. \\ Finally $ \partial^M(q'^{down} _{k}) =   \partial'(q'^{down} _{k})$ if $\alpha (q'^{down} _{k}, q^{up} _{k-1}) $ is zero; if not, we would have two closed orbits $O$ and $O'$ in $V$ such that the faces of $O$ are connected to $O'$ by some V-paths and their indices differ by two which is not possible.  \\ Therefore on $C_*(V)$,  $\partial^M=  \partial' $ and $\partial'^ 2 = 0$ since by Morse-Floer theory for combinatorial  gradient vector fields  $(\partial^M)^2 = 0$. 
      \item[3.] We define  $\varphi * : C_*(V) \longrightarrow C_*(V')$ as
        follows. For $0\le k \le m$, 
     \begin{eqnarray*}
	\varphi_*( p_k)= p_k,\quad \varphi_*  (O^0_ {k}) = q'^{down}_k ,\quad \varphi_*( O^1_ {k-1})= q^{up}_k 
 \end{eqnarray*}      
     
	$\varphi_* $ is an isomorphism by the  above partitioning method for the set of rest points of $V'$. To prove $\varphi_*$ is a chain map  from  $C_*(V) $ to $ C_*(V')$ we should have $ \partial' \varphi_* = \varphi_*   \partial  $

Here, we show the equality for $O^1_ {k-1}$ and for the other generators of
$C_k(V)$  it can be similarly obtained:. 

\begin{equation*}
\begin{aligned}[t]
 \varphi_* \partial ( O^1_ {k-1}) 
    &=  \varphi_* (\sum \alpha( O^1_{k-1},  O^1_{k-2}) . O^1_{k-2}  \\
    &= \sum \alpha (q^{up} _{k}, q^{up} _{k-1}) . q^{up} _{k-1} \\
    &= \partial'q^{up} _{k}\\
    &= \partial' \varphi_* ( O^1_ {k-1}) 
  \end{aligned}
  \end{equation*}
   
In the second equality, $ \varphi_* $ preserves the parity of $\alpha( O^1_{k-1},  O^1_{k-2})$ since each equivalence class of $SVP(O^1_ {k-1} , O^1_ {k-2})$  corresponds to exactly one gradient V-path from  $q^{up} _{k}$ to $q^{up} _{k-1}$ (and one gradient V-path from  $q'^{down} _{k-1}$ to $q'^{down} _{k-2}$).  \\
Therefore  $\partial' \varphi_* = \varphi_*   \partial $ and  since $\partial^{'2} = 0$  and  $\partial^ 2 =\varphi_*^{-1} \partial^{'2} \varphi_*, \partial^ 2 = 0$
  \end{proof} 
  
  	  \begin{remark}  
For the three types of attachments in the above equality, after replacing orbits
with two rest simplices and two gradient V-paths between them, $\alpha(p_k,
q^{up}_{k-1})$, $ \alpha (q^{up} _{k}, q^{up} _{k-1})$, $\alpha (q'^{down}
_{k}, q'^{down} _{k-1})$ and $\alpha (q'^{down}
_{k}, p _{k-1})$ are also one. For instance, in the left tetrahedron
below, we have two orange rest simplices, one of index two $B_2$ at the bottom
and one of index zero at the top $T_0$ and two red and purple closed orbits of
indices one and zero. If we convert the red orbit into two rest simplices, marked with red crosses, one of index two $R_2$ and the other one of index one $R_1$ and similarly turn the purple orbit into two rest simplices, marked with purple crosses, one of index one $P_1$ and the other one of index zero $P_0$ (shown in the right figure), we have 
$ \alpha (B_2, P_1)=1$, $\alpha(R_2, P_1) = 1$, $\alpha (R_1, P_0) =1$ and $\alpha(R_1, T_0) = 1 $ (Also $\alpha(R_2, R_1) = 0 $, $\alpha (P_1, P_0) =0$).   \\

\begin{center}
\hspace{-0.2 cm}\vspace{-0.01 cm} \includegraphics[width=3.5 cm]{th.jpg} 
\hspace{0.4 cm}\vspace{-0.01 cm} \includegraphics[width=3.5 cm]{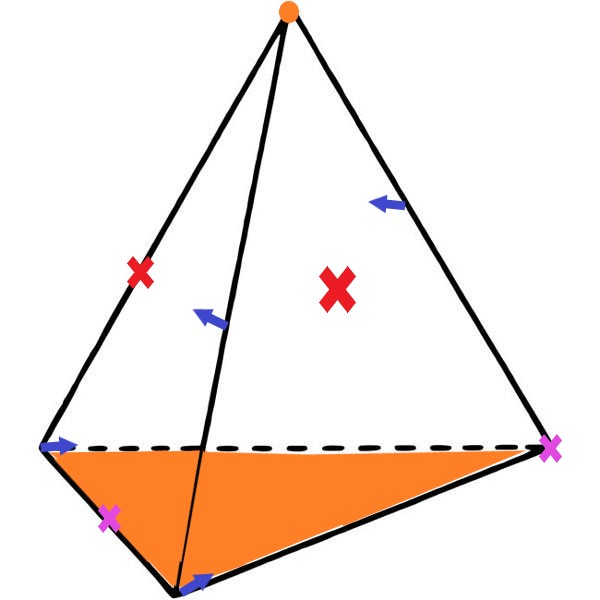} 
\end{center}
  	  
   \end{remark}    

We can then define the $\mathbb{Z}_2$ Morse-Floer homology of $M$, for each
$k, 0 \leq k \leq m $ by 
 \begin{center}
$H_k (M, \mathbb{Z}_2) = \frac {ker (\partial_k) }{image (\partial_{k+1})}$
\end{center}

\medskip

In  \cite{Forman2}  Forman proved Morse inequalities for general combinatorial
vector fields based on his combinatorial Morse type Lyapunov function. There
the main components are rest points and orbits in basic sets. Here we want to
present these inequalities in a much shorter way based on the idea of changing
every orbit of index $k-1$, $O_{k-1}$ to two rest points of index $k$ and
$k-1$ as above. The following result can be considered as the combinatorial
version of what Franks proved for  smooth Morse-Smale dynamical systems  \cite{Franks}. 
    \begin{theorem}
Let $V$ be a combinatorial vector field over a finite simplicial complex $M$ with $c_k$ rest points of index $k$ and $A_k$ orbits of index $k$. Then
$$c_k-c_{k-1}+....\pm c_0 + A_k \geq \beta_k- \beta {k-1}+... \pm \beta_0, $$
where $\beta_k = \dim H_k(M, \mathbb{Z}_2)$.
   \end{theorem}
     \begin{proof}
    We create a new vector field $V'$ over $M$ by replacing each closed orbit
    with two rest points as above. Since there is no closed orbit in $V'$,
    based on what Forman showed in \cite{Forman}, $V'$ is the gradient of some
    combinatorial Morse function on $M$. On the other hand, the indices of
    rest points do not change when turning $V$ to $V'$ and  $V'$ has $c'_k$
    rest points of index $k$ where $c'_k = c_k+ A_k+ A_{k-1} $. Applying the
    Morse inequalities for gradient vector fields  to $c'_k$ in $V'$  gives us the desired inequalities. 
     \end{proof}
     
    \begin{remark}  
 If a simplicial complex is obtained by triangulation of a non-orientable
manifold, we might not get the correct ($ \mathbb{Z}_2$) homology groups when
the chain recurrent set of our combinatorial vector field has non-stationary
closed V-paths. 
However for computing the $\mathbb{Z}_2$ homology, we can turn each orbit into two rest simplices and two V-paths between
them as above to get a combinatorial gradient vector field on the simplex and
use the classical discrete Floer-Morse theory. For instance, consider a triangulation of the Klein bottle which
has two closed orbits represented by red and blue arrows of
index one and zero,  respectively,  as shown in the left diagram below. We note that the red orbit is twisted. Here by turning
orbits into two rest simplices and two V-paths between them we get the correct
$\mathbb{Z}_2 $-homology of the Klein bottle  which is the same as the  $\mathbb{Z}_2$-homology of the triangulated torus as  $\mathbb{Z}_2$-homology cannot distinguish between orientable and non-orientable surfaces.

\begin{center}
\hspace{-1 cm}\vspace{-0.01 cm} \includegraphics[width=3.5 cm]{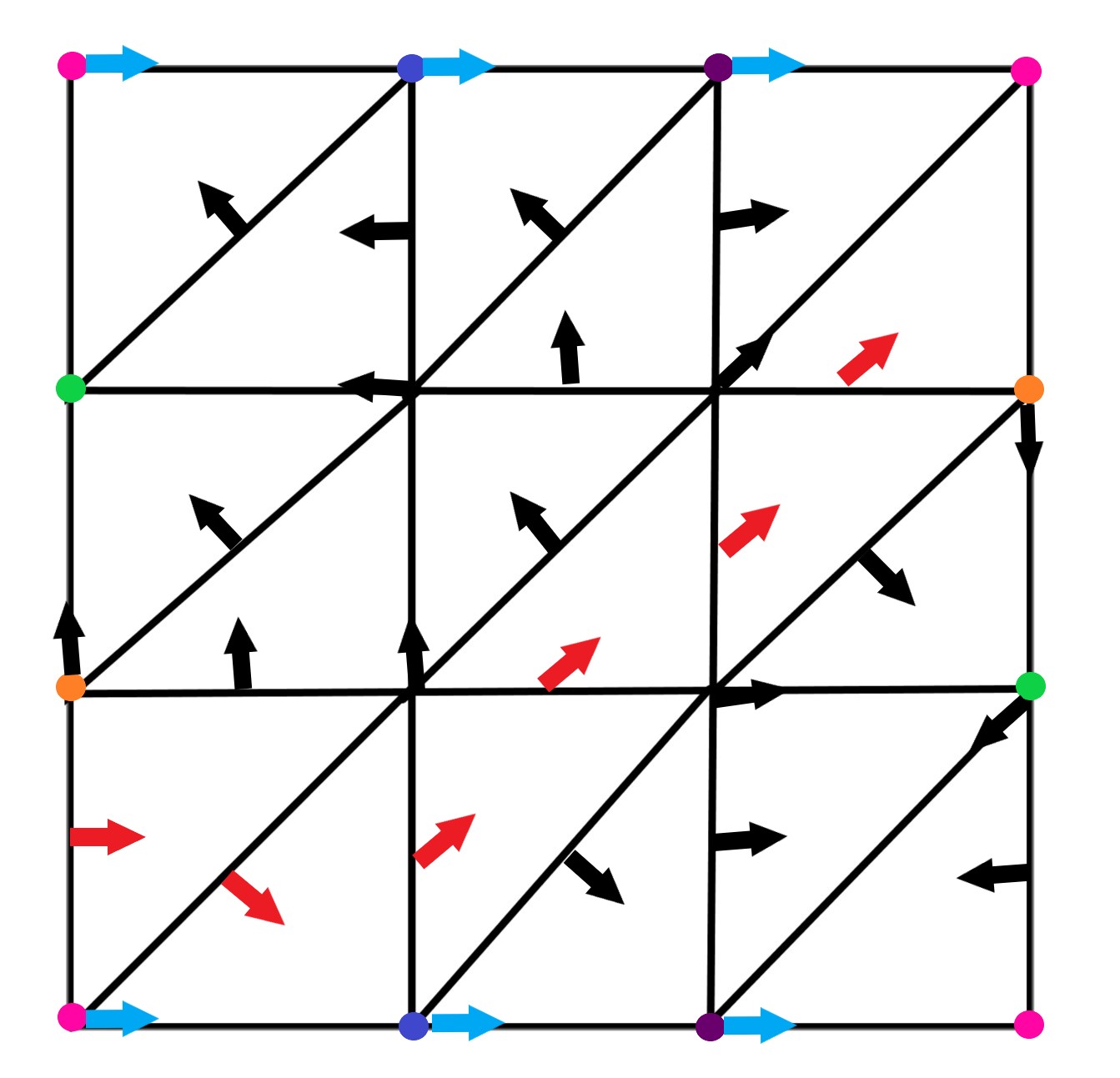} 
\hspace{1 cm}\vspace{-0.01 cm} \includegraphics[width=3.5 cm]{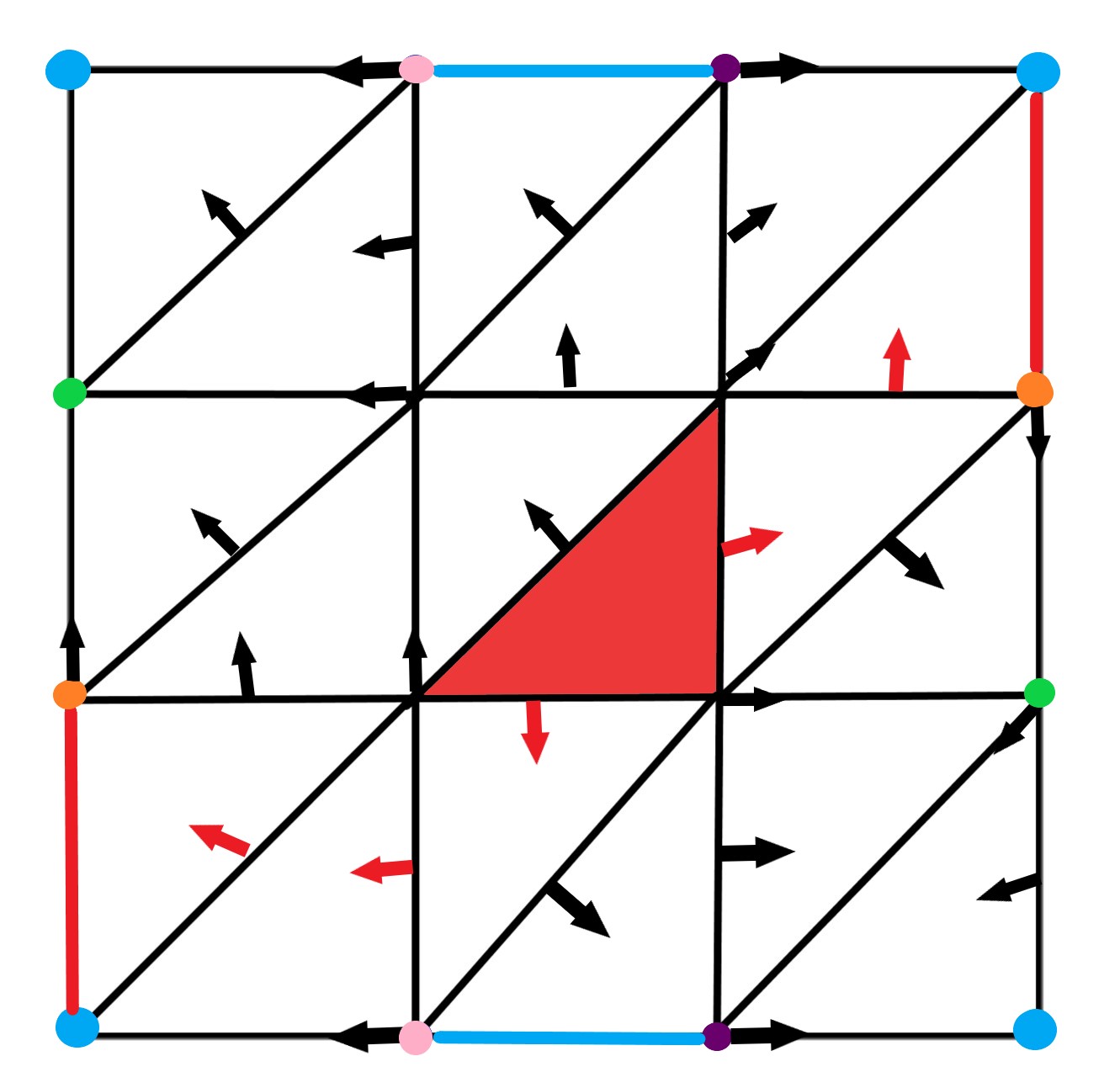} 
\end{center}
  	  
   \end{remark}   
         	  \begin{remark}  
  Orientability of a simplicial complex is not a necessary condition for
  defining the Floer boundary operator for  general combinatorial vector
  fields. For instance in the following diagram, we have a non-orientable
  simplex as each vertex is the intersection of three different edges.  Here,
  the chain recurrent set has an orbit of index zero with  red arrows, as well as three critical edges and one critical vertex in the middle.  If we use theorem \ref{2} to compute the $\mathbb{Z}_2$ homology of $M$ based on this vector field we get $H_0(M,  \mathbb{Z}_2)=1$ and $H_1(M,  \mathbb{Z}_2)=3$.    	  

\begin{center}
 \includegraphics[width=3.5 cm]{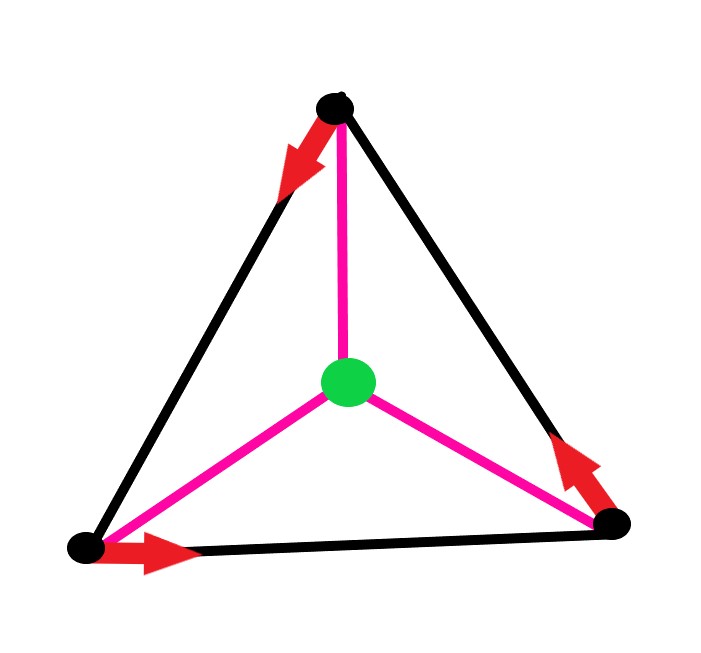} 
\end{center}
  	  
   \end{remark}   
        \subsection{Computing Homology Groups of Simplicial Complexes}
We now present the computation of Floer homology groups for some CW complexes.\\

1. Consider the tetrahedron as a symmetric triangulation of the Sphere $S ^ 2$, 
equipped with a vector field $V$ which has two rest simplices, one of index
(dimension) zero $(p_0)$ (the vertex at the top corner) and another of index
two $\tau_2$ (the simplex at the bottom), shown in orange, and a closed red orbit  $O_1$  of index one and one of index zero $O'_0$ in purple.  \\
 \begin{wrapfigure}{r}{0.3\textwidth}
\hspace{-0.1 cm}\vspace{0.4 cm} \includegraphics[width=3.5 cm]{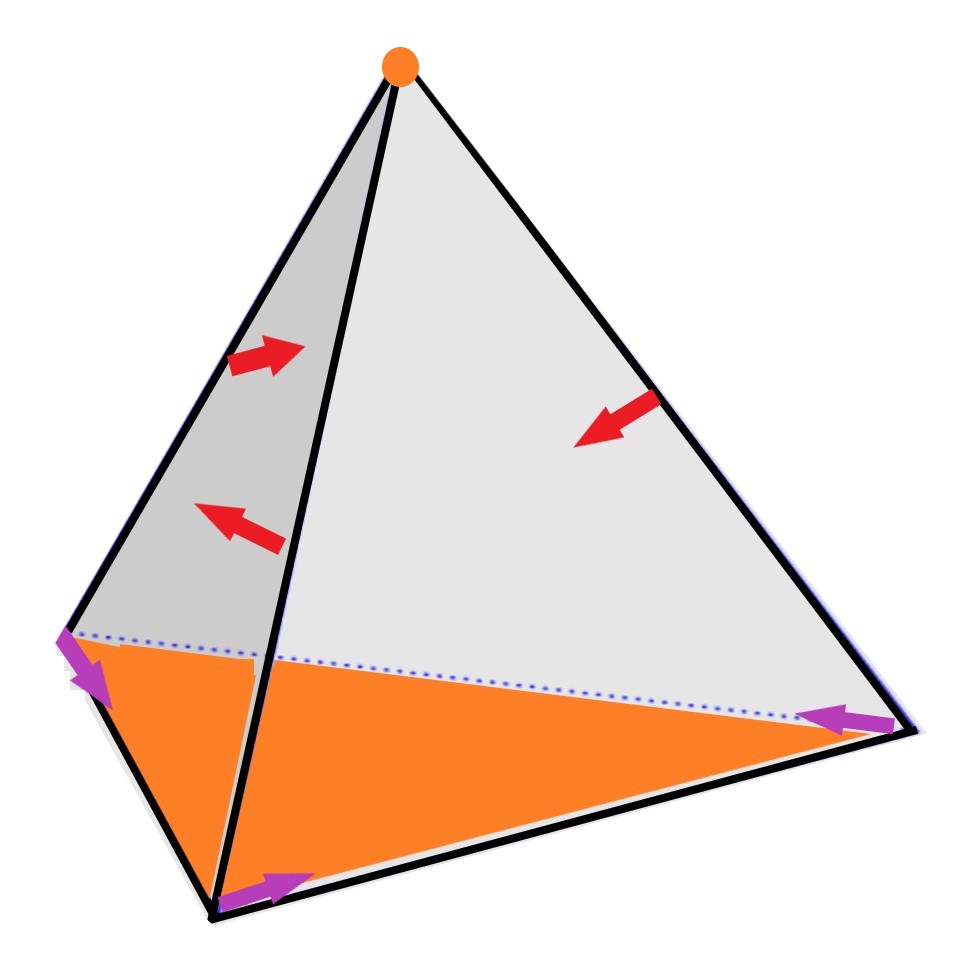} 
\end{wrapfigure}

\begin{eqnarray*}
  C_2&=& \left( \tau_2, O^1_{1}\right)\\
  C_1&=& \left(  O^0_ {1}, O'^1_ {0}
         \right) \\
  C_0&=& \left( O'^0_ {0}, p_0 \right)
  \end{eqnarray*}

$\partial_2 \tau_2= O'^1_ {0}$ as this is an attachment where the faces of $\tau_2$ and edges in $ O^1_ 0$ coincide. On the other hand, $\partial_2 O^1_{1}$ is also equal to $O'^1_ {0}$ (another type of attachment) and therefore $\tau_2- O^1_{1} $ is the only generator for $H_2(M, \mathbb{Z}_2)$.\\
$\partial_1 O^0_ {1}=  \alpha ( O^0_ {1}, O'^0_ {0}) .O'^0_ {0} + \alpha ( O^0_ {1}, p_{0}) .p_{0} = O'^0_ {0} + p_{0} \neq 0 $ and $O^0_ {1}$ does not contribute to $H_1(M, \mathbb{Z}_2)$. Also $\partial_1 O'^1_ {0}= 0 $ but since $O'^1_ {0}$ is in the image of $\partial_2$, it does not contribute to $H_1(M, \mathbb{Z}_2)$ and $H_1(M, \mathbb{Z}_2)=0 $ 
Finally $\partial_0  p_0= 0=\partial_0 O'^0_ {0} $, but since $O'^0_ {0} + p_{0}$ is in the image of $\partial_1$  we have just one generator for  $H_0(M, \mathbb{Z}_2)$. 

 \begin{wrapfigure}{r}{0.3\textwidth}
\hspace{-0.01 cm}\vspace{-1 cm} \includegraphics[width=6 cm]{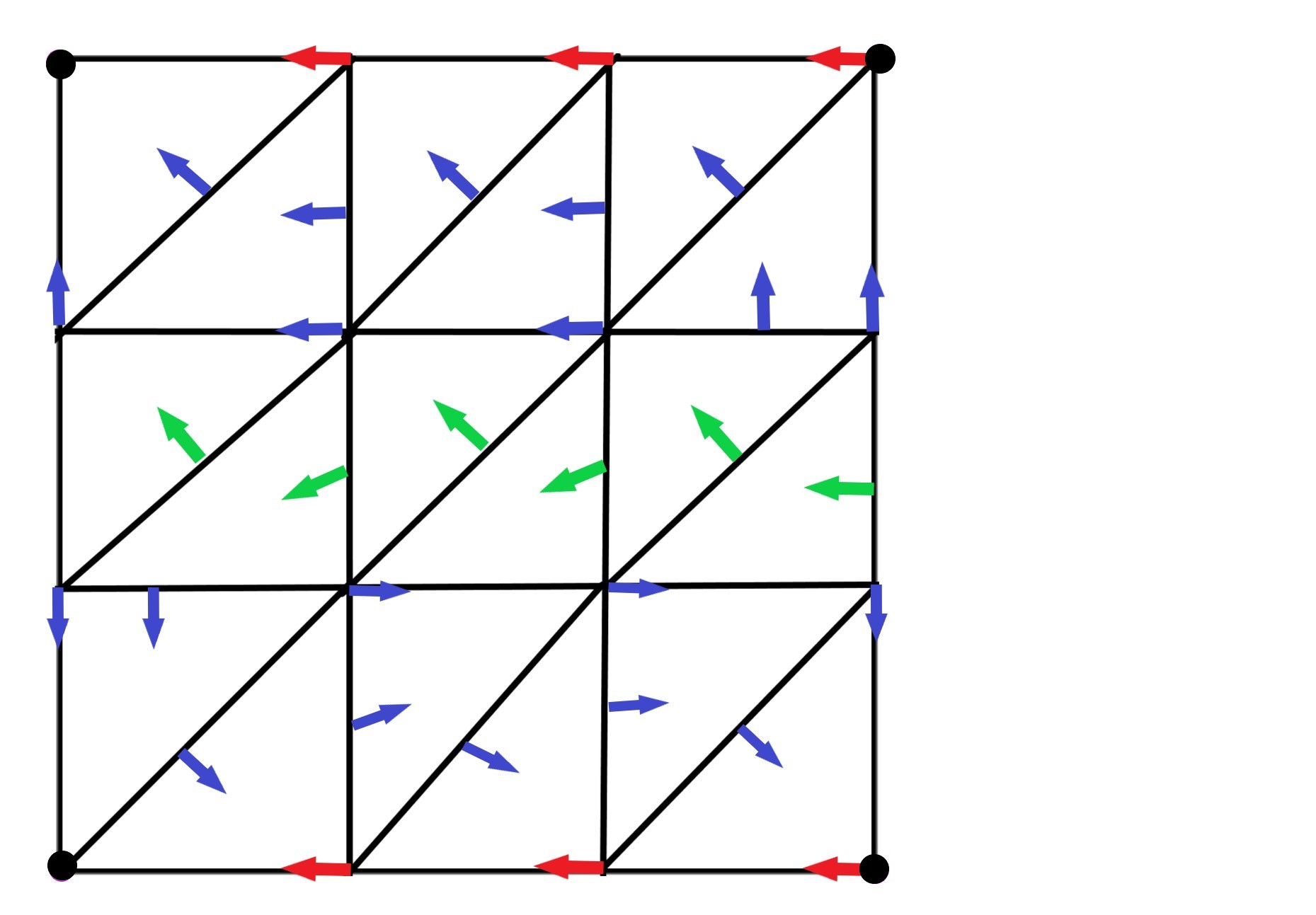}
\end{wrapfigure}
2. Let $T ^ 2$ at right be a triangulation of the two dimensional torus equipped with a  vector field which has two closed orbits $O_1$ and $O'_0$ with green and red arrows. \\
Since this case is actually a discrete version of example 4 in the previous section, we have analogous structures for the chain complexes and boundaries; $SVP(O^1_1 , O'^1_ 0) $ is partitioned into two equivalence classes and therefore  
 \begin{wrapfigure}{r}{0.3\textwidth}
\hspace{-0.01 cm}\vspace{0.1 cm} \includegraphics[width=6 cm]{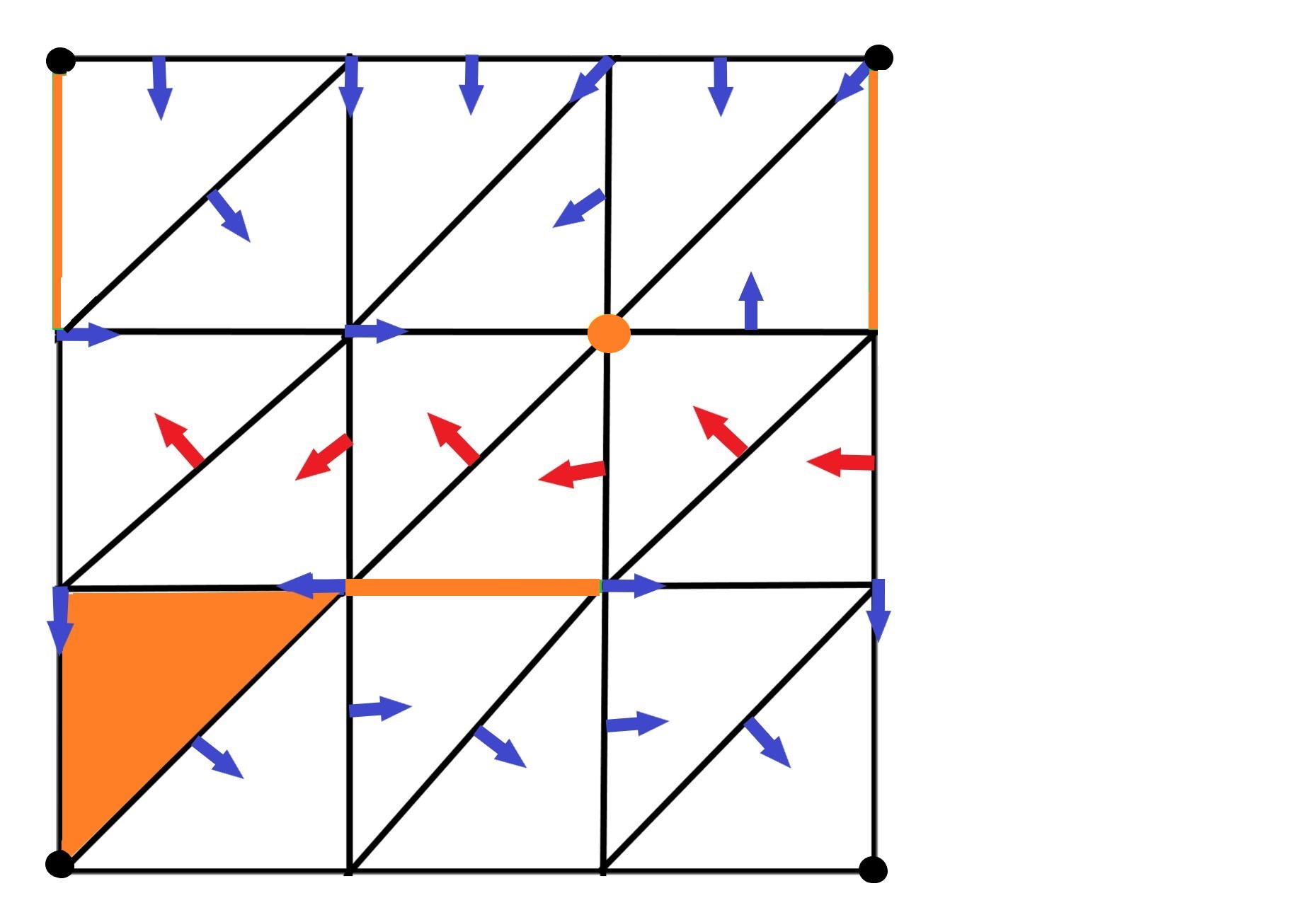} 
\end{wrapfigure} 
$\alpha ( O^1_1,  O'^1_0)= \alpha( O^0_1,  O^0_0)= 0  $ and all of the generators of $C_k$ for $k=0, 1, 2$ contribute to the corresponding homology groups.

3. Consider another combinatorial vector field on the triangulated
torus where $V$ has four orange rest simplices, one of
index zero $(p_0)$, a vertical edge $ve_1$ of index one, a horizontal edge $he_1$ of
index one and one rest simplex $\tau_2$  of index two and also a red orbit $O$ of index one. We have
\begin{eqnarray*}
  C_2&=& \left( \tau_2, O^1_{1}\right)\\
  C_1&=&   \left( ve_1,
         he_1,  O^0_ {1} \right) \\
  C_0&=& \left( p_o \right)
  \end{eqnarray*}

$\partial_2 \tau_2 = 1. ve_1 + 1. he_1 \neq 0$;  also $\partial_2 O^1_{1}=0 $ as there is no orbit of index zero here. Therefore $O^1_{1}$ is the only generator for $H_2(M,  \mathbb{Z}_2)$.
$\partial_1 ve_1 = 2. p_0 =0 =\partial_1 he_1 $ but since  $ve_1 + he_1$ is in the image of $\partial_2$,  $ve_1 - he_1$ is one generator for $H_1(M,  \mathbb{Z}_2)$. $\partial_1  O^0_ {1}= 2. p_0 =0 $ and therefore $O^0_ {1}$ is the other generator of $H_1(M,  \mathbb{Z}_2)$. $\partial_0  p_0= 0$ and it is the generator for $H_0(M,  \mathbb{Z}_2)$.\\\\
4. Finally to compute the Floer homology groups of the depicted cube,
 we consider a vector field $V$ that has two (orange and yellow) rest simplices
 of index two  at the top $\tau^N_2$ and at the bottom $\tau^S_2$ and three different orbits, one blue orbit $(bO)_0$ of index zero, one green orbit $(gO)_0$ of index zero and a red orbit $(rO)_1$ of index one.
 \begin{figure}[h!] 
 \centering
\includegraphics[width=4 cm]{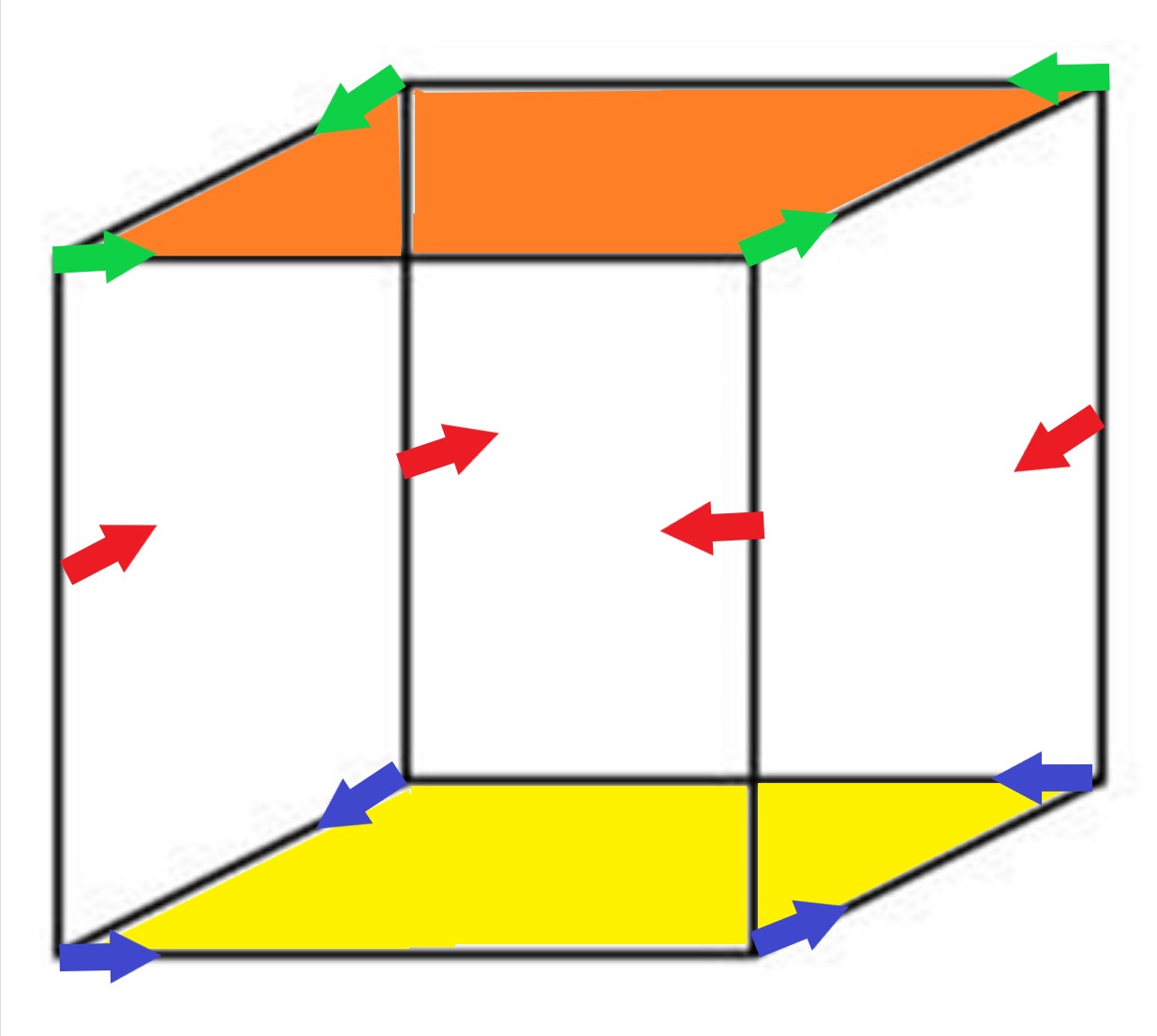} 
 \end{figure}
\begin{eqnarray*}
  C_2&=& \left(\tau^N_2, \tau^S_2,  (rO)^1_{1}\right) \\
  C_1&=& \left(  (rO)^0_ {1}, (bO)^1_ {0}, (gO)^1_ {0} \right) \\
  C_0&=& \left(  (bO)^0_ {0}, (gO)^0_ {0}  \right) 
\end{eqnarray*}
$\partial_2 \tau^N_2 = 1. (gO)^1_ {0} \neq 0$; $\partial_2 \tau^S_2 = 1. (bO)^1_ {0} \neq 0$; $\partial_2(rO)^1_{1}= 1. (gO)^1_ {0}  + 1. (bO)^1_ {0} \neq 0$  but $\partial_2 ( \tau^N_2 + \tau^S_2 - (rO)^1_{1})= 0 $ and therefore we have one generator for  $H_2(M,  \mathbb{Z}_2)$. \\ 
$\partial_1 (rO)^0_ {1}= 1. (bO)^0_ {0} + 1. (gO)^0_ {0} \neq 0$. $ \partial_1 (bO)^1_ {0} = 0= \partial_1 (gO)^1_ {0} $, but since both $(bO)^1_ {0}$ and $(gO)^1_ {0} $ are in the image of $\partial_2$, we have no generator for $H_1(M,  \mathbb{Z}_2)$. Finally  $\partial_0  (bO)^0_ {0}= 0 = \partial_0  (gO)^0_ {0} $, but as $(bO)^0_ {0} + (gO)^0_ {0}$ is in the image of  $ \partial_1$ we have one single generator for $H_0(M,  \mathbb{Z}_2)$.

\section{Acknowledgments}
The authors would like to thank Clemens Bannwart and Claudia Landi for pointing out a problem with the original formulation of  Theorem 2.9, as mentioned in  the Remark 2.11. The first author thanks C\'edric De Groote, Parvaneh Joharinad and Rostislav Matveev for their enlightening comments/questions, on the first draft, which helped to improve the manuscript.
\bibliographystyle{plain}
\bibliography{Morse-Floer}
\end {document}